\documentclass[12pt]{amsart}
\usepackage{verbatim,amsmath,amssymb,mathabx}
\usepackage[active]{srcltx}
\usepackage[all]{xy}
\usepackage{graphicx,hyperref}

\newtheorem{thm}{Theorem}[section]
\newtheorem{prop}[thm]{Proposition}
\newtheorem{lem}[thm]{Lemma}

\newtheorem{cor}[thm]{Corollary}

\newtheorem*{Mthm}{Main Theorem}

\newtheorem*{propD}{Property D}

\theoremstyle{definition}
\newtheorem{dfn}[thm]{Definition}

\def\C{\mathbb{C}}

\def\R{\mathbb{R}}

\def\Q{\mathbb{Q}}
\def\0{\emptyset}
\def\Ac{\mathcal{A}} \def\Bc{\mathcal{B}} \def\Cc{\mathcal{C}} \def\Dc{\mathcal{D}}
  
 \def\Mc{\mathcal{M}} 
\def\Oc{\mathcal{O}} \def\Pc{\mathcal{P}} 
\def\Sc{\mathcal{S}} 
   
 \def\Zc{\mathcal{Z}}  

\def\Z{\mathbb{Z}}
\renewcommand\emptyset{\varnothing}
\newcommand{\sm}{\setminus}

\def\eps{\varepsilon}
\def\ol{\overline}
\def\d{\partial}
  \def\ta{\theta}  \def\Ga{\Gamma}
\def\al{\alpha}    \def\la{\lambda}

\def\le{\leqslant}
\def\ge{\geqslant}

\def\uc{\mathbb{S}^1}
\def\disk{\mathbb{D}}
\def\<{\langle}
\def\>{\rangle}

\def\d{\partial}

\def\cdisk{\ol\disk}
\def\thu{\mathrm{TH}}

\def\Re{\mathrm{Re}}

\begin{document}
\date{Feb 15, 2021; revised on Oct 28 and Dec 3, 2021}

\title[Modeling core parts of Zakeri slices I]
{Modeling core parts of Zakeri slices I}

\dedicatory{To the memory of A.M. Stepin}

\author[A.~Blokh]{Alexander~Blokh}

\author[L.~Oversteegen]{Lex Oversteegen}

\author[A.~Shepelevtseva]{Anastasia Shepelevtseva}

\author[V.~Timorin]{Vladlen~Timorin}

\address[A.~Blokh, L.~Oversteegen]
{Department of Mathematics\\ University of Alabama at Birmingham\\
Birmingham, AL 35294-1170}

\address[A.~Shepelevtseva, V.~Timorin]
{Faculty of Mathematics\\
HSE University, Russian Federation\\
6 Usacheva St., 119048 Moscow
}

\address[A.~Shepelevtseva]
{Scuola Normale Superiore, 7 Piazza dei Cavalieri, 56126 Pisa, Italy}

\address[Vladlen~Timorin]
{Independent University of Moscow\\
Bolshoy Vlasyevskiy Pereulok 11, 119002 Moscow, Russia}

\email[Alexander~Blokh]{ablokh@math.uab.edu}
\email[Lex~Oversteegen]{overstee@uab.edu}
\email[Anastasia~Shepelevtseva]{asyashep@gmail.com}
\email[Vladlen~Timorin]{vtimorin@hse.ru}



\begin{abstract}
The paper deals with cubic 1-variable polynomials whose Julia sets are connected.
Fixing a bounded type rotation number, we obtain a slice of such polynomials with the origin being a fixed Siegel point
 of the specified rotation number.
Such slices as parameter spaces were studied by S. Zakeri, so we call them \emph{Zakeri slices}.
We give a model of the central part of a slice
(the subset of the slice that can be approximated by hyperbolic polynomials with Jordan curve Julia sets),
 and a continuous projection from the central part to the model.
The projection is defined dynamically and agrees with the dynamical-analytic parameterization
 of the Principal Hyperbolic Domain by Petersen and Tan Lei.
\end{abstract}

\maketitle

\section{Introduction}
In this introduction we assume a certain level of familiarity with complex dynamics; detailed
definitions will be given later on.

For a polynomial $P$ denote by $[P]$ its affine conjugacy class. By the
\emph{degree $d$ polynomial parameter space} one understands the space
of such classes of polynomials of degree $d$. Similarity between
quadratic dynamical planes and slices of parameter spaces of higher degree
polynomials is a recurring topic of research. A now standard mechanism
(found in \cite{BuHe}) uses holomorphic renormalization. If, say, a
cubic polynomial $P$ is \emph{immediately renormalizable} (i.e., has a
connected quadratic-like filled Julia set $K^*(P)$), then one critical
point of $P$ belongs to $K^*(P)$. The other critical point of $P$ may
eventually map to $K^*(P)$ in which case $P$ belongs to a quasiconformal
copy of $K^*(P)$ contained in the parameter space of cubic polynomials.
A more general renormalization scheme established in \cite{IK12} allows
to find copies of
 $\Mc_2\times\Mc_2$ (where $\Mc_2$ is the quadratic Mandelbrot set) or
 $\mathcal{MK}$ (the set of pairs $(c,z)$, where $c\in\Mc_2$,
 and $z$ belongs to the filled Julia set $K(P_c)$ of $P_c(w)=w^2+c$)
in the cubic connectedness locus.
In the non-renormalizable case, things are much subtler.

Suppose that a cubic polynomial $P$ has a non-repelling fixed point $a$.
It can always be arranged by a suitable affine conjugacy that $a=0$;
one can consider this point as marked and, hence, instead of affine conjugacies work with
linear conjugacies $A(z)=\al z$, where $\al\in\C\sm\{0\}$, that leave $0$ fixed.
Much is known if $P$ is renormalizable;
 this case, under the additional assumption
 that $P$ tunes a hyperbolic polynomial, is considered in \cite{IK12,sw20}.
The remaining, non-renormalizable case, needs closer attention.
Consider the set of all affine conjugacy classes $[P]$ of cubic polynomials $P$ with $P(0)=0$ and $|P'(0)|\le 1$.
A central part of this parameter space, analogous to the interior of the main cardioid, is the \emph{principal hyperbolic component}
 consisting of classes $[P]$ for all hyperbolic $P$ with $|P'(0)|<1$ and Jordan curve Julia set.
An analytic parametrization of the principal hyperbolic component with
dynamical meaning is given in \cite{PT09} where the authors were able
to describe pieces of the boundary of the principal hyperbolic component
 contained in the locus of classes $[P]$ with $|P'(0)|<1$.
 This paper aims at a similar description in the Siegel case under the assumption that
the associated rotation number has bounded type.

A powerful method of studying polynomials with non-repelling periodic
points is based upon linearizations.
Consider a polynomial $f$ with attracting or neutral fixed point $a$
(we discuss polynomials, but a lot of the results are in fact more general).
A \emph{linearization} is a holomorphic map $\psi$ of an open disk $\disk(r)$ of radius $r>0$ around $0$
such that $\psi(0)=a$, and $\psi(\la z)=f\circ\psi(z)$ for all $z\in\disk(r)$ where $\la=f'(a)$.
Assume that $r>0$ is the radius of convergence of the power series of $\psi$ at $0$.
It is known that $\psi:\disk(r)\to\C$ is an embedding, cf. \cite{che20}.
Then $\psi(\disk(r))$ is called the \emph{linearization domain} $\Delta(f,a)$ of $f$ around $a$.
If $|\la|<1$, then $\Delta(f, a)$ is compactly contained in the attracting basin of $a$,
 and $\partial \Delta(f, a)$ contains a critical point.
In the case $a=0$, the domain $\Delta(f,a)$ is denoted by $\Delta(f)$.

Fix $\la$ with $|\la|\le 1$.
Let $\C_\la$ be the space of complex linear conjugacy classes of complex cubic polynomials
with fixed point $0$ of multiplier $\la$
(alternatively, $\C_\la$ consists of affine conjugacy classes of cubic polynomials with
\emph{marked} fixed point of multiplier $\la$).
For a cubic polynomial $P(z)=\la z+\dots$, let $[P]_0$ be its class \emph{in $\C_\la$}.
Write $\Cc_\la\subset\C_\la$ for the \emph{connectedness locus} in $\C_\la$.
That is, $[P]_0\in\Cc_\la$ if the Julia set $J(P)$ of $P$ is connected.
A central part of $\Cc_\la$ is the set $\Pc_\la$ of all $[P]_0\in\Cc_\la$ that lie in the closure of the principal hyperbolic component.
We are interested in understanding the topology and combinatorics of $\Pc_\la$
 through a comparison with a suitable dynamical object.

As the basis for comparison, consider the space of quadratic polynomials $Q(z)=Q_\la(z)=\la z(1-z/2)$.
Then $\la$ is the multiplier of the fixed point $0$ of $Q$.
Suppose that either $|\la|<1$ or $\la=e^{2\pi i\ta}$, where $\ta\in\R/\Z$ is of bounded type.
Let $\psi=\psi_{Q}:\disk\to\Delta(Q)$ be the corresponding linearization (here $\disk=\disk(1)$).
The set $\ol \Delta(Q)$ is a Jordan disk --- this is a classical result of Douady--Ghys--Herman--Shishikura,
 see \cite{Dou87,Her87,Sw98}.
Therefore, the Riemann map extends to a homeomorphism $\ol\psi:\ol\disk\to\ol \Delta(Q)$.
The finite critical point of $Q$ is 1, thus the linearizatiton domain
$\Delta(Q)$ around $0$ contains $1$ in its boundary.
We normalize $\psi$ so that $\ol\psi(1)=1$.
If $|\la|=1$, then the map $\ol\psi$ conjugates the rigid rotation by angle $\ta$ with the restriction of $Q$ to $\ol\Delta(Q)$.
Consider the quotient $\tilde K(Q)$ of the set $K(Q)\sm\Delta(Q)$ by the equivalence relation $\sim$ defined as follows.
Two different points $z$, $w$ are equivalent if both belong to $\d\Delta(Q)$, and $\Re(\ol\psi^{-1}(z))=\Re(\ol\psi^{-1}(w))$.

There is a partially defined correspondence --- stated as Property D below ---
 between the dynamical plane of $P$ and that of $Q$.
Recall that a continuous map $\eta:X\to Y$ between two compacta is said to be \emph{monotone} if,
 for every connected subset $B\subset Y$, the set $\eta^{-1}(B)\subset X$ is connected.
In order to verify that $\eta$ is monotone, it suffices to check that all point preimages are connected.

\begin{propD}
For any cubic polynomial $P$ with $[P]_0\in\Pc_\la$, there exist a full $P$-invariant continuum $X(P)$
 containing both critical points of $P$ and a continuous map $\eta_P:X(P)\to K(Q)$ that
 semi-conjugates $f|_{X(P)}$ with $Q|_{\eta_P(X(P))}$.
If both critical points of $P$ are in the Julia set, then the map $\eta_P$ is monotone.
\end{propD}

The letter $D$ in \emph{Property D} stands for ``Dynamics'' (or ``Douady'').
This property will be used to, quoting Douady,
 \emph{``seed in the dynamical plane and reap the harvest in the parameter plane''}.

\begin{Mthm}
Suppose that $\theta\in\R/\Z$ is of bounded type, and $\lambda=e^{2\pi i\theta}$.
Let $Q=Q_\lambda$ be a quadratic polynomial with a fixed point of multiplier $\lambda$.
Then there is a continuous map $\Phi_\la:\Pc_\la\to \tilde K(Q)$ taking $[P]_0$ to the
 $\eta_P$-image of some critical point of $P$.
\end{Mthm}

The map $\Phi_\la$ is illustrated in Figure \ref{fig:pla}.

\begin{figure}
  \centering
  \includegraphics[height=4cm]{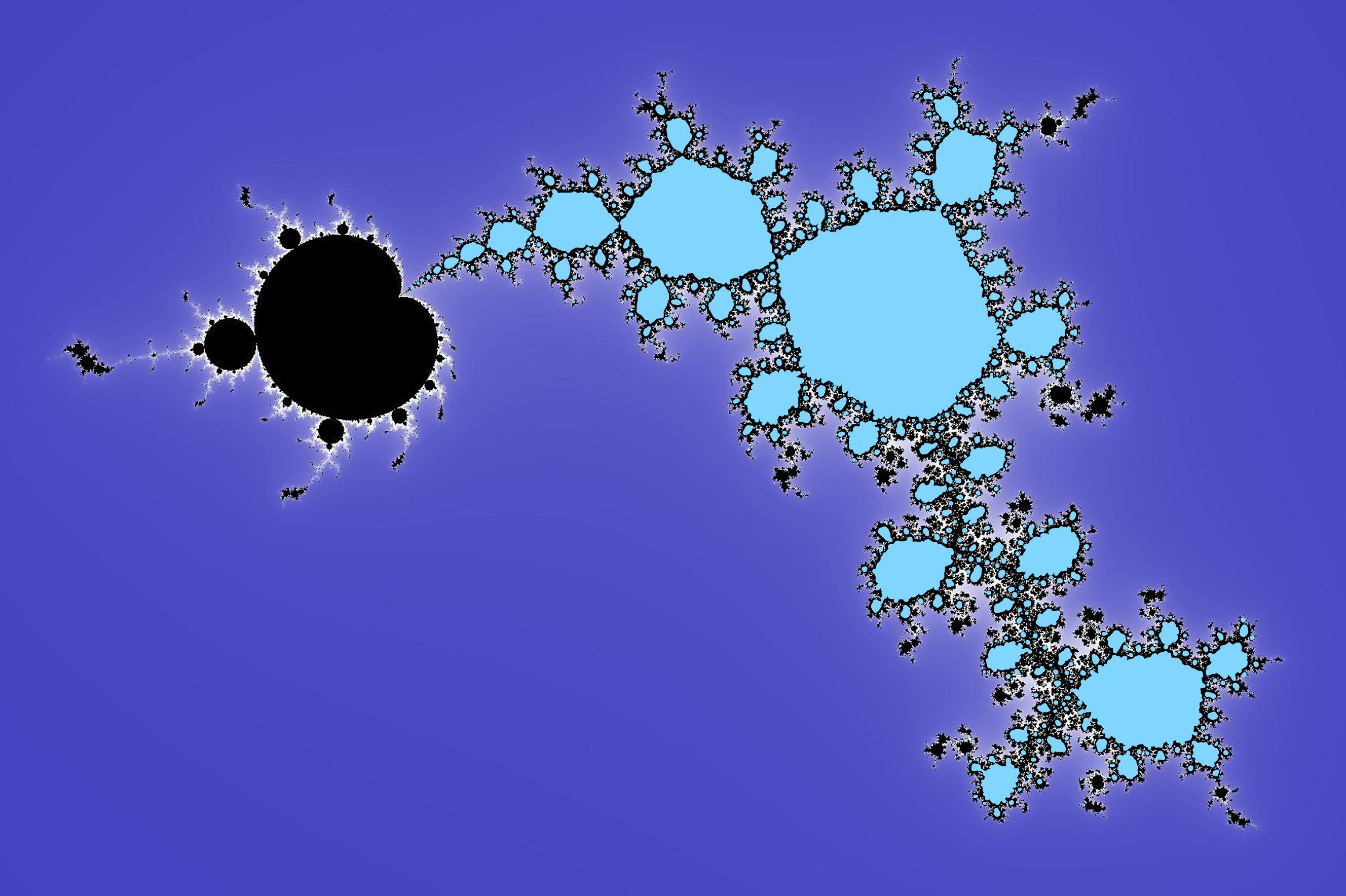}
  \includegraphics[height=4cm]{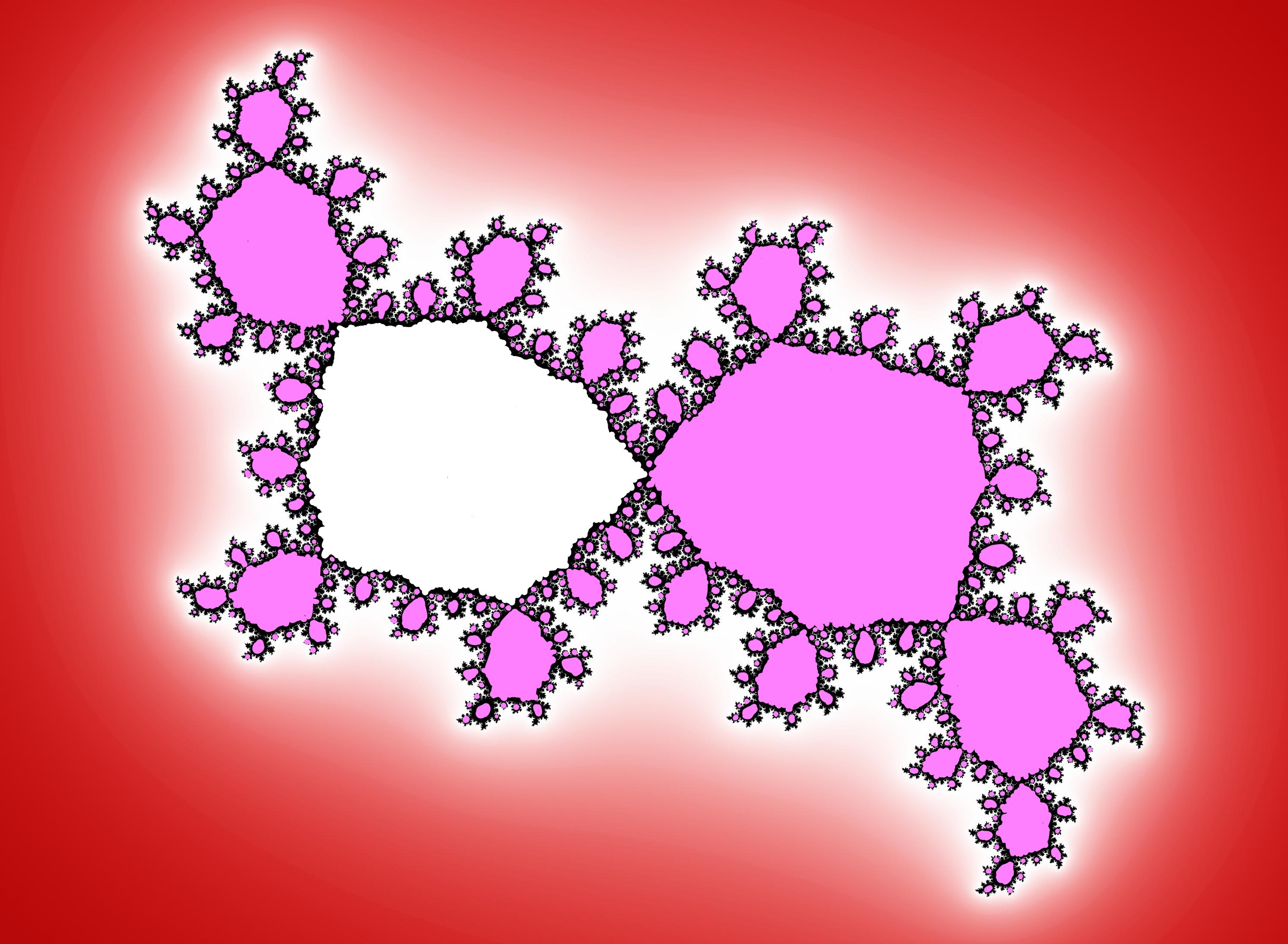}
  \caption{\small Left: the parameter plane $\C_\la$ with $\la=\exp(\pi i\sqrt{2})$.
  We used the parameterization, in which every linear conjugacy class from $\C_\la$
  is represented by a polynomial of the form $f(z)=\la z+\sqrt{a}z^2+z^3$,
  where $a$ is the parameter (that is, the figure shows the $a$-plane).
  The conjugacy class of $f$ is independent on the choice between the two values
  of the square root.
  Regions with light uniform shading are interior components of $\Pc_\la$.
  There are also various ``decorations'' of $\Pc_\la$ (that is, components of $\Cc_\la\sm\Pc_\la$)
  shown in black; these decorations contain copies of the Mandelbrot set.
  Right: the dynamical plane of $Q=Q_\la$.
  The bounded white region near the center is the Siegel disk $\Delta(Q)$.
  A conjectural model of $\Pc_\la$ is obtained from $K(Q)$ by removing this white region and
  gluing its boundary into a simple curve.
  Our main theorem provides a continuous map from $\Pc_\la$ to this conjectural model.
  }\label{fig:pla}
\end{figure}

In this paper we do not address the issue of $\Phi_\la$ being surjective or monotone ---
a discussion of these properties is postponed to a later publication.
It can be observed that the Main Theorem is a direct (partial) extension of \cite{PT09}.
According to \cite{PT09}, C. Petersen, P. Roesch and Tan Lei planned a continuation
 that should have contained an analog of the above Main Theorem for parabolic slices.
Apparently, this continuation never appeared in print.

Observe also that while this paper concentrated on the set $\Pc_\la$,
 the structure of the entire parameter $\la$-slice $\C_\la$ with the corresponding $\la$-slice of the connectedness locus
 $\Cc_\la\subset\C_\la$ was studied in \cite{slices}.

\section{Background and a specification of the Main Theorem}
\label{s:bcg}
Take $\la=e^{2\pi i\ta}$, where $\ta\in\R\sm\Q$.
Let $p_n/q_n$ be the sequence of rational approximations of $\ta$ based on the continued fraction expansion.
By the Brjuno--Yoccoz theorem \cite{Brj71,yoc},
 a holomorphic germ $f$ with $f(0)=0$ and $f'(0)=\la$ is linearizable at $0$ if and only if
$$
\sum_{n=1}^\infty\frac{\log q_{n+1}}{q_n}<\infty.
$$
The latter condition is called the \emph{Brjuno condition};
if $\ta$ satisfies it,  $\ta$ is said to be a \emph{Brjuno number}.
So, if $\ta$ is a Brjuno number, then a polynomial $f$ with $f(0)=0$ and $f'(0)=\la$
 has a Siegel disk $\Delta(f)$ around $0$.
Say that $\ta$ is \emph{bounded type} if the continued fraction coefficients of $\ta$ are bounded.
Any bounded type irrational number is Brjuno; the converse is not true. 
From now on and throughout the paper, we set $\la=e^{2\pi i\ta}$ and assume that $\ta$ is bounded type.

Let us discuss the parameter slices that are of interest in this paper, and how to parameterize them.
Start with the quadratic case.
Recall that a quadratic polynomial with a fixed point of multiplier $\la$ is unique up to an affine conjugacy.
We use the normalization $Q_\la(z)=\la z(1-z/2)$ with the property that $0$ is the fixed point of multiplier $\la$
and the finite critical point of $Q_\la$ is $1$.
Some known results about $J(Q_\la)$ are summarized in the following theorem.

\begin{thm}
  \label{t:Pe}
  The Julia set $J(Q_\la)$ is locally connected and has zero Lebesgue measure.
Moreover, the Siegel disk $\Delta(Q_\la)$ is a quasidisk.
\end{thm}

The second part of Theorem \ref{t:Pe} is in fact a theorem of Douady--Ghys--Herman--Shishikura
 (a proof can be obtained as a combination of \cite{Dou87} and a Theorem
 of Herman and \'Swi\c{a}tek  \cite{Her87,Sw98}).
The first part is Theorem A of \cite{Pe}.

To parameterize cubic polynomials with fixed point at $0$ and marked critical points,
we work with the space $\C^*_\la$ of polynomials
$$
P_{c, \la}(z)=P_c(z)=\la z\left(1-\frac{1}{2}\left(1+\frac{1}{c}\right)z+\frac{1}{3c}z^2\right)
$$
(we often fix $\la$ and then omit it in the notation).
The parameter $c$ is chosen so that $P_c$ has critical points $1$ and $c$, and
$\la$ is the multiplier of the fixed point $0$. It is easy to verify that
$P_c$ and $P_{1/c}$ are linearly conjugate and, hence, $[P_c]_0=[P_{1/c}]_0$.
The map $c\mapsto P_c$ establishes an isomorphism between $\C^*=\C\sm\{0\}$ and $\C^*_\la$.

If $|\la|<1$, then $0$ is an attracting fixed point of $P_c$.
Let us now fix $\la=e^{2\pi i\ta}$, where $\ta\in\R/\Z$ is of bounded type,
and describe the results of \cite{Za} where this case was studied in great detail.
First, the disk $\Delta(P_c)$ is non-degenerate, at least one critical point of $P_c$
still belongs to $\partial \Delta(P_c)$, and $\Delta(P_c)$ is a quasidisk.
Let $\Zc^c_\la$ be the set $\{P_c\in\C^*_\la\mid \{c, 1\}\subset\d\Delta(P_c)\}$.
The set $\Zc^c_\la$, called the \emph{Zakeri curve}, is a Jordan curve.
It divides the punctured plane $\C^*_\la$ into two components, $\Oc^*_\la(0)$ and $\Oc^*_\la(\infty)$,
 each isomorphic to the punctured disk $\disk\sm\{0\}$.
The corresponding punctures are $c=0$ and $c=\infty$, respectively.

Since $P_c$ and $P_{1/c}$ are linearly conjugate, the involution $c\mapsto 1/c$
 interchanges the punctured disks $\Oc^*_\la(0)$ and $\Oc^*_\la(\infty)$ and maps $\Zc^c_\la$ to itself.
Observe that $P_1$ and $P_{-1}$ always belong to $\Zc^c_\la$.
Moreover, the following holds:
\begin{enumerate}
  \item if $c\in \Oc^*_\la(0)$ then $c\in \d \Delta(P_c)$ and $1\notin \partial \Delta(P_c)$,
  \item if $c\in \Zc^c_\la$ then $c, 1\in \d \Delta(P_c)$, and
  \item if $c\in \Oc^*_\la(\infty)$ then $c\notin \d\Delta(P_c)$ and $1\in \partial \Delta(P_c)$.
\end{enumerate}
Every class in $\C_\la$ is represented by polynomials $P_c$, $P_{1/c}\in \C^*_\la$ (with suitable $c$).
Thus, the space $\C_\la$ identifies with a quotient of $\C^*_\la$.
The corresponding quotient map $\tau$ identifies $P_c$ with $P_{1/c}$.
It restricts to homeomorphisms on $\Oc^*_\la(0)$ and on $\Oc^*_\la(\infty)$ and folds $\Zc^c_\la$ to a simple arc $\Zc_\la$.
Moreover, the quotient projection is two-to-one on $\Zc^c_\la$ except two points $P_1$ and $P_{-1}$.
The space $\C_\la$ can be described as $\tau(\Oc^*(\infty)\cup\Zc^c_\la)$; this description will be often used in the sequel.

Recall that $\Cc_\la\subset \C_\la$ is the connectedness locus in $\C_\la$;
write $\Cc^c_\la\subset\C^*_\la$ for the corresponding connectedness locus in $\Oc^*_\la(\infty)\cup\Zc^c_\la$.
In other words, $\Cc^c_\la$ consists of polynomials $P\in\Oc^*_\la(\infty)\cup\Zc^c_\la$ such that $K(P)$ is connected.
The superscript ``c'' in the notation $\Cc^c_\la$ means that $c$ is the free critical point
 (the other critical point $1$ is associated with the Siegel point $0$).
More generally, this superscript appears in the notation of a parameter space object if this object belongs to
(or is contained in) $\C^*_\la$
 (in particular, critical points are marked), and $c$ can be regarded as a free critical point.
Note that $\Zc^c_\la\subset \Cc^c_\la$ as for $c\in \Zc^c_\la$ we have $c, 1\in \d\Delta(P_c)$
(and, hence, both critical points of $P_c$ are non-escaping).
The set $\Cc_\la$ coincides with the image of $\Cc^c_\la$
 under the quotient map $\tau$.
We want to describe the structure of the connectedness loci $\Cc^c_\la$ and $\Cc_\la$.

\begin{thm}[\cite{Za}]
\label{t:Za-qs}
  If $P\in\Cc^c_\la$, then $1\in \d \Delta(P)$ and $\d \Delta(P)$
is a quasicircle
depending continuously on $P\in\Cc^c_\la$ in the Hausdorff metric.
\end{thm}

Define the set $\Pc^c_\la$ as the subset of $\Cc^c_\la$ consisting of polynomials that can be approximated by
 sequences $P_{n}\in\Cc^c_{\la_n}$ with $|\la_n|<1$ and both critical points of $P_n$ in the immediate basin of $0$.
This is the central part of $\Cc^c_\la$ which we want to model.
It is easy to see that $\Pc^c_\la$ is a compactum containing $P_1$
(indeed, polynomials $P_{r\la, 1}$ converge to $P_1$ as $r\nearrow 1$ and, on the other hand,
are such that both critical points belong to the immediate basin of $0$).

\begin{lem}
\label{l:pc-conn}
The sets $\Pc^c_\la$ and $\Pc_\la$ are connected.
\end{lem}

\begin{proof}
The quotient projection $\tau$ from $\C^*_\la$ to $\C_\la$ is a branched 2-1 covering
 with the only branch points at $P_{\pm 1}\in\Pc^c_\la$.
Therefore, connectedness of $\Pc^c_\la$ is equivalent to connectedness of $\Pc_\la$.
For every $\la'\in\disk(1)$, the principal hyperbolic component $\mathcal{H}^c_{\la'}$ of $\C^*_{\la'}$
 is defined as the set of all hyperbolic $P\in\C^*_{\la'}$ such that $J(P)$ is a Jordan curve.
It follows from \cite{PT09} (and the fact that branch points of $\tau:\C^*_{\la'}\to\C_{\la'}$ lie
 in $\mathcal{H}^c_{\la'}$) that $\mathcal{H}^c_{\la'}$ is connected.
Now consider a sequence $P_{(n)}\to P\in\Pc^c_\la$, where $P_{(n)}\in\Pc^c_{\la_n}$ and $|\la_n|<1$.
It follows that $\la_n\to\la$.
It suffices to find a connected subset of $\Pc^c_\la$ containing both $P$ and $P_1$.
Then every point $P\in\Pc^c_\la$ is connected to $P_1$, hence $\Pc^c_\la$ is connected.

Passing to a subsequence, we may assume that continua $\ol{\mathcal{H}^c_{\la_n}}$ converge in the Hausdorff metric.
Moreover, each $\mathcal{H}^c_{\la_n}$ contains a unicritical (i.e., with a multiple critical point in $\C$)
 polynomial $P_{1,(n)}\to P_1$.
The limit continuum then contains both $P$ and $P_1$, as desired.
\end{proof}

Since $\la$ is fixed, it can be omitted from the notation of $Q=Q_\la$.
We will define a continuous map $\Phi^c_\la$ from $\Pc^c_\la$ to the model space $K(Q)\sm\Delta(Q)$.
Note that, in the case of marked critical points, the model space is simpler as we do not pass to a quotient.
This map is conjecturally a homeomorphism.
As often happens in holomorphic dynamics,
 the definition of $\Phi^c_\la$ depends on a certain map between dynamical planes.
More precisely, we will define a $P$-invariant continuum $X(P)\subset K(P)$ and a continuous map $\eta_P:X(P)\to K(Q)$
 such that $\eta_P\circ P=Q\circ\eta_P$ on $X(P)$.
The following theorem makes the Main Theorem more specific.

\begin{thm}
  \label{t:main1}
  The map $\eta_P:X(P)\to K(Q)$ is monotone for every $P\in\Cc^c_\la$ except when $c\in X(P)\sm J(P)$.
For every $P_c\in\Pc^c_\la$, the critical point $c$ is in $X(P_c)$.
The map $\Phi^c_\la:P_c\mapsto \eta_{P_c}(c)$ is defined and continuous on $\Pc^c_\la$.
It takes values in $K(Q)\sm\Delta(Q)$.
\end{thm}

The Main Theorem follows directly from Theorem \ref{t:main1} by applying the quotient projection $\tau$ from $\C^*_\la$ to $\C_\la$.

\subsection*{Plan of the paper}
In Section \ref{s:bub}, the principal tools of this paper are developed.
These include bubbles, legal arcs, and Siegel rays.
To an extent, Siegel rays compensate for the absence of repelling cutpoints in the central part of $K(P_c)$
 (here $P_c\in \Cc^c_\la$).
They form a controllable combinatorial structure, map forward in a regular way,
 and divide the central part of $K(P)$ into smaller pieces.
Section \ref{s:stab} discusses the issues of stability.
Roughly speaking, a dynamically defined set is stable if it moves continuously as we change the parameters.
Outside of the Zakeri curve, stability can be defined in customary language of holomorphic motions.
However, since $\Cc^c_\la$ is not a Riemann surface at points of the boundary curve $\Zc^c_\la$,
 we need to consider a more general notion of an equicontinuous motion.
The principal results of Section \ref{s:stab} claim that Siegel rays are stable.
Section \ref{s:dynmap} deals with the dynamical map $\eta_P:X(P)\to K(Q)$ defined on the central part $X(P)$ of $K(P)$.
In particular, property D is established for this map.
Finally, Section \ref{s:par} concludes the proof of Theorem \ref{t:main1} and the Main Theorem.

\section{Bubbles}
\label{s:bub}

\subsection{An overview of \cite{Za}}
\label{ss:petzak}
Consider \emph{Blaschke fractions}, i.e., products of \emph{generalized Blaschke factors} $\frac{z-p}{1-\ol pz}$
 \textbf{without} assuming that $|p|<1$ (in a classical Blaschke factor, $p$ has to belong to $\disk$).
The Blaschke fractions have some common properties with the classical Blaschke products.
In particular, they have the inversion self-conjugacy:
 if $B(z)$ is a Blaschke fraction, then $\frac{1}{\ol{B(z)}}=B(\frac{1}{\ol z})$, i.e.,
 the inversion $z\mapsto \frac{1}{\ol z}$ conjugates $B$ with itself.
It follows that critical points of $B$ are split in two groups: critical points inside $\cdisk$ and their inversions with respect to
$\uc$ that are located outside $\disk$.

In order to describe the structure of $\Cc^{c}_\la$, Zakeri introduced an auxiliary family of
 degree 5 Blaschke fractions given by
$$
B(z)=e^{2\pi i t} z^3\left(\frac{z-p}{1-\ol pz}\right)\left(\frac{z-q}{1-\ol qz}\right),
$$
where $|p|>1, |q|>1$. In addition to the inversion self-conjugacy (and, hence, the inversion symmetry of their critical points)
the restrictions of these maps on $\uc$ (which is invariant) are homeomorphisms.
Indeed, by the Argument Principle, the topological degree of $B:\uc\to\uc$ is equal to the number of zeros
 minus the number of poles (counting multiplicities) of $B$ in $\disk$.
The latter number is $3-2=1$ (triple zero at $0$ and simple poles at $1/\ol p$ and $1/\ol q$).

Zakeri chooses $p$ and $q$ so that $B$ has a multiple critical point in
$\uc$ and two critical points $c_B$, $1/\ol c_B$ that may or may not
belong to $\uc$. The angle $t\in\R/\Z$ is adjusted so that
$B:\uc\to\uc$ has rotation number $\ta$. Consider the Blaschke products
$B$ as above, with marked critical points and normalized (via
conjugation by a rigid rotation) so that one of the critical points is
$1$. Then the space $\Bc_\la$ of all such $B$'s is parameterized by a
single complex parameter $\mu\in\C\sm\disk$ (recall that $\la=e^{2\pi
i\ta}$; thus the dependence on $\ta$ is expressed through $\la$) such
that
the critical points of $B_{\mu}$ are $\mu$ and $1$.
Note also that $B_\mu$ and $B_{1/\mu}$ are linearly conjugate for $\mu\in\uc$
but we distinguish them as elements of $\Bc_\la$.

By a theorem of Herman and \'Swi\c{a}tek \cite{Sw98},
the map $B:\uc\to\uc$ is quasi-symmetrically (qs) conjugate to a rigid rotation.
Consider a qc-extension $H=H_B:\ol\disk\to\ol\disk$ of this quasi-symmetric conjugacy
(take the Douady--Earle extension \cite{de86} to make the construction unique), and define the \emph{modified Blaschke product}
$\tilde B$ as $B(z)$ for $|z|\ge 1$ and $H^{-1}\circ \mathrm{Rot}_\ta\circ H(z)$ for $|z|<1$.
Here $\mathrm{Rot}_\ta$ is the rigid rotation about $0$ by angle $\ta$.
Finally, $\tilde B$ is shown to be qc conjugate to a cubic polynomial $P\in\Cc^{c}_\la$ by finding
a certain $\tilde B$-invariant conformal structure $\sigma$ on $\ol\C$, and straightening it.
Here the critical point $c$ of $P$ corresponds to the critical point $\mu$ of $B$.
Define the \emph{non-escaping locus} of $\Bc_\la$ as the set of $B_\mu\in\Bc_\la$ for which the orbit of $\mu$ is bounded.
Set $P=\Sc(B)$; the map $\Sc$ from the non-escaping locus of $\Bc_\la$ to $\Cc^{c}_\la$ is called the \emph{surgery map}.
The following proposition is Corollary 10.5 of \cite{Za}.

\begin{prop}
  \label{p:C10.5Za}
  There is an equicontinuous family of qc homeomorphisms $\varphi_B:\C\to\C$ parameterized by $B$
  in the non-escaping locus of $\Bc_\la$
such that $\Sc(B)=\varphi_B\circ\tilde B\circ\varphi_B^{-1}$,
and normalized so that $\varphi_{B}(1)=1$.
\end{prop}

Set $P=\Sc(B)$. Note that the Siegel disk $\Delta(P)$ of $P$ equals
$\varphi_B(\disk)$,
 and the Riemann map $\psi_{\Delta(P)}:\disk\to\Delta(P)$ coincides with $\varphi_B\circ H_B^{-1}$.
All $H_B$ are quasi-conformal with the same qc constant that depends only on $\ta$.
It follows (cf. Theorem 4.4.1 of \cite{hub}) that $H_B$ and $H_B^{-1}$ form an equicontinuous family.
We obtain the following corollary.

\begin{cor}
  \label{c:equicont}
The extended Riemann maps $\ol\psi_{\Delta(P)}$ (where $P$ varies through $\Cc^{c}_\la$) form an equicontinuous family.
\end{cor}

Let $C(\ol\disk,\C)$ be the space of all continuous maps from
$\ol\disk$ to $\C$ with the sup-norm. Corollary \ref{c:equicont}, in
turn, implies the following.

\begin{cor}
  \label{c:sup-conv}
The map from $\Cc^c_\la$ to $C(\ol\disk,\C)$ taking $P$ to $\ol\psi_{\Delta(P)}$ is continuous.
\end{cor}

\begin{proof}
  Suppose that a sequence $P_n\in\Cc^c_\la$ converges to $P\in\Cc^c_\la$.
We want to prove that $\ol\psi_n=\ol\psi_{\Delta(P_n)}$ converge uniformly to $\ol\psi=\ol\psi_{\Delta(P)}$.
First note that $\ol\psi_n(1)=\ol\psi(1)=1$ by the chosen normalization of the Riemann maps.
Now fix a positive integer $k$ and consider the point $z=P^k(1)\in\d\Delta(P)$.
Clearly, for fixed $k$ and all large $n$, the points $z_n=P_n^k(1)\in\d\Delta(P_n)$ are close to $z$.
Since $\ol\psi_n$ conjugate the rotation by angle $\ta$ with $P_n|_{\ol\Delta(P_n)}$,
 we necessarily have $\ol\psi_n(e^{2\pi i k\ta})=z_n$.
Similarly, $\ol\psi(e^{2\pi i k\ta})=z$.
Thus $\ol\psi_n\to\ol\psi$ point-wise on a dense subset of $\uc$.
By equicontinuity, it follows that $\ol\psi_n\to\ol\psi$ uniformly on $\uc$.
Finally, by the Maximum Modulus Principle, $\ol\psi_n\to\ol\psi$ on $\ol\disk$.
\end{proof}

\subsection{Polar coordinates and bubbles}
Let $U\subset\C$ be an open topological disk equipped with a distinguished \emph{center} $a\in U$,
 a certain \emph{radius} $r_U\in (0,\infty)$ and a \emph{base point} $b\in\d U$ accessible from $U$.
An open topological disk $U$ equipped with these data is called a \emph{framed domain}.
These data constitute a \emph{framing} of $U$.
For a framed domain $U$, consider the Riemann map $\psi_U:\disk(r_U)\to U$ such that $\psi(0)=a$ and
 $\lim_{u\to r_U}\psi_U(u)=b$ with $u$ converging to $r_U$ radially.
If $U$ is a Jordan disk, then $\psi_U$ extends to a homeomorphism $\ol\psi_U:\ol\disk(r_U)\to\ol U$.

\begin{dfn}[Polar coordinates, internal rays]
  \label{d:polar}
  Let $U$ be a framed domain with center $a\in U$, root point $b$, and radius $r_U$.
A point $z\in U$ has the form $\psi_U(\rho_z e^{2\pi i\ta_z})$ for some $\rho_z\in [0,r_U)$ and $\ta_z\in\R/\Z$.
The \emph{polar radius function} is by definition the function $z\mapsto \rho_z$ on $U$.
We always extend this function (keeping the notation) to $\ol U$ by setting $\rho_z=r_U$ for all $z\in\d U$.
The \emph{polar angle function} is by definition the function $z\mapsto \ta_z$ on $U\sm\{a\}$.
Note that this function is undefined when $\rho_z=0$.
If $U$ is a Jordan disk (and only in this case), we extend the polar angle function to $\ol U$ by continuity.
Then, for $z\in\d U$, the angle $\ta_z$ is determined by the relation $z=\ol\psi_U(r_U e^{2\pi i\ta_z})$.
Given any $\al\in\R/\Z$ define the \emph{internal ray} $R_U(\al)$ as the set $\{z\in U\mid \ta_z=\al\}$.
Say that $R_U(\al)$ \emph{lands} at a point $w\in\d U$ if $w$ is the only point in $\ol R_U(\al)\sm U$.
If $U$ is a Jordan disk, then every internal ray $R_U(\al)$ lands at the point $\ol\psi_U(r_U e^{2\pi i\al})$.
\end{dfn}

Assume now that either $f:\C\to\C$ is in $\Cc^c_\la$, or $f=Q_\la$.
Recall that $\la=e^{2\pi i\ta}$ is fixed.
Write $\Delta(f)$ for the Siegel disk of $f$, and $\psi_f:\disk\to\Delta(f)$ for the Riemann map normalized so that $\psi_f(0)=0$ and $\ol\psi_f(1)=1$; recall that $1$ is a critical point of $f$.

Define a \emph{pullback} of a connected set $A\subset\C$ under a polynomial $f$ as a connected component of $f^{-1}(A)$.
\emph{An iterated pullback} of $A$ under $f$ is by definition an $f^n$-pullback of $A$ for some $n>0$.

\begin{dfn}[Bubbles and polar coordinates on bubbles]\label{d:bub-coor}
\emph{Bubbles} \newline of $f$ are iterated pullbacks of $\Delta(f)$
(thus, bubbles are open Jordan disks).
Let $A$ be a bubble of $f$, and let $n$ be the smallest integer with $f^n(A)=\Delta(f)$.
Such $n$ is called the \emph{generation} of $A$ and denoted by $\mathrm{Gen}(A)$.
For $z\in A\sm f^{-n}(0)$,
set $\ta_z=\ta_{f^n(z)}-n\ta$ and $\rho_z=\rho_{f^n(z)}$.
Now, if $z$ has polar coordinates $\rho$ and $\al$, then $f(z)$ has polar coordinates $\rho$ and $\al+\ta$.
Equivalently, the complex coordinate $\rho e^{2\pi i\al}$ multiplies by $\la$ under the action of $f$.
Note that the polar radius extends as a continuous function on the union of the closures of all bubbles.
\end{dfn}

If a bubble $A$ is a homeomorphic iterated pullback of $\Delta(f)$, then we define a framing of $A$ as follows.
The \emph{center} $o_A$ of $A$ is defined as the only iterated preimage of $0$ in $A$.
If $f^n(A)=\Delta(f)$, then the base point of $\d A$ is defined as the point $b_A$ with $f^n(b_A)=f^n(1)$.
With this framing, internal rays of $A$ are defined.
By definition, an internal ray of $A$ consists of all points with a fixed value of polar angle.
Then in $A$ there is one internal ray of a given polar argument,
 and all internal rays connect the center of $A$ with appropriate points on $\d A$.
In particular, this picture holds for all bubbles in the quadratic case.

In the case of a cubic polynomial $P_c$ there might be a bubble $B$ which contains $c$ and, hence, maps forward two-to-one. In that case the picture with polar angle function, the center and the internal rays is a bit different.
More precisely, if $c$ maps into the center of $P_c(B)$, then the pullbacks of an internal ray are two
 internal rays connecting $c$ with appropriate points on $\d B$.
Now, suppose that $P_c(c)$ is not the center of $P_c(B)$.
Then the center $x=o_{P_c(B)}$ of $P_c(B)$ has two preimages $x'$, $x''\in B$.
Hence, with one exception, each internal ray of $P_c(B)$ pulls back to two internal rays,
 each connecting the appropriate pullback of the center of $P_c(B)$ with the appropriate point on $\d B$.
The exception is the internal ray $R=\ol{xy}$, of argument, say, $\al$, passing through $P_c(c)$;
 its pullback is a ``cross'' with endpoints at $x'$, $x''$ and at two preimages of $y$, and with vertex at $c$.

Given $(\rho, \al)$, there is unique point $z\in \Delta(f)$ and
a lot of other points with coordinates $(\rho, \al)$.
Any point with polar coordinates $(\rho,\al)$ maps to $f^n(z)$ under $f^n$, for some $n$ depending on the point.

The terms ``bubbles'' and ``bubble rays'' were introduced in the Thesis of J. Luo \cite{Luo95}
 (cf. \cite{AY09,Yan17} for a development of these ideas).
However, the difference with our setup is that bubbles in the sense of Luo are Fatou components
 that are eventually mapped to a superattracting rather than a Siegel domain.
Also, similar ideas are used in \cite{bbco10} where some
quadratic Cremer Julia sets were studied by approximating them with Siegel Julia sets with specific properties.

If $f^n:A\to \Delta(f)$ is a conformal isomorphism, then it also defines a framing of $A$ so that
the polar coordinates on $A$ just defined are consistent with this framing. By an \emph{oriented arc}
we mean an arc $I$ whose one endpoint is marked as \emph{initial} and the other
is marked as \emph{terminal}.

Let $A$ be a bubble of a cubic polynomial $P_c\in\Cc^c_\la$.
Evidently, a point $z\in \ol{A}$ can be connected with $0$ by an arc $I\subset K_P$
with initial point $0$ and terminal point $z$. While
such an arc is not unique, it is easy to see that for any bubble $A$ the intersection $I\cap \ol{A}$ is a subarc of $I$.
In what follows we will only consider arcs such that for all of the bubbles involved (except possibly for one)
the intersection $I\cap \ol{A}$ is contained in the union of $\{o_A\}$ and the closures of two internal rays of $\ol{A}$.
More precisely, let us now define \emph{legal arcs}.

\begin{dfn}[Legal arcs]\label{d:reg-arc}
Consider an oriented topological arc $I\subset K(f)$.
Suppose that $I^\circ$ is an open dense subset of $I$ such that the following holds:

\begin{enumerate}

\item the set $I\sm I^{\circ}$ can accumulate only at the terminal point of $I$;

\item each component of $I^\circ$ is contained in one bubble $A$ and coincides with
 a component of $(A\sm\bigcup_{n\ge 0} f^{-n}(0))\cap I$.

\item the polar angle function is defined and constant on each component of $I^\circ$;

\item $P^n(I)$ is not separated by $c$ for $n\ge 0$.

\end{enumerate}
Then $I$ is called a \emph{legal arc} (see Fig. \ref{fig:reg}).
Let $\al_0$, $\dots$, $\al_k$, $\dots$ be the values of the polar angle on $I^\circ$ taken in the order they appear on $I^\circ$.
A linear order of $\al_i$s is well defined since $I$ is oriented.
The finite or infinite sequence $(\al_0,\dots,\al_k,\dots)$ is called the \emph{(polar) multi-angle} of $I$.
\end{dfn}

\begin{figure}
  \centering
  \includegraphics[width=8cm]{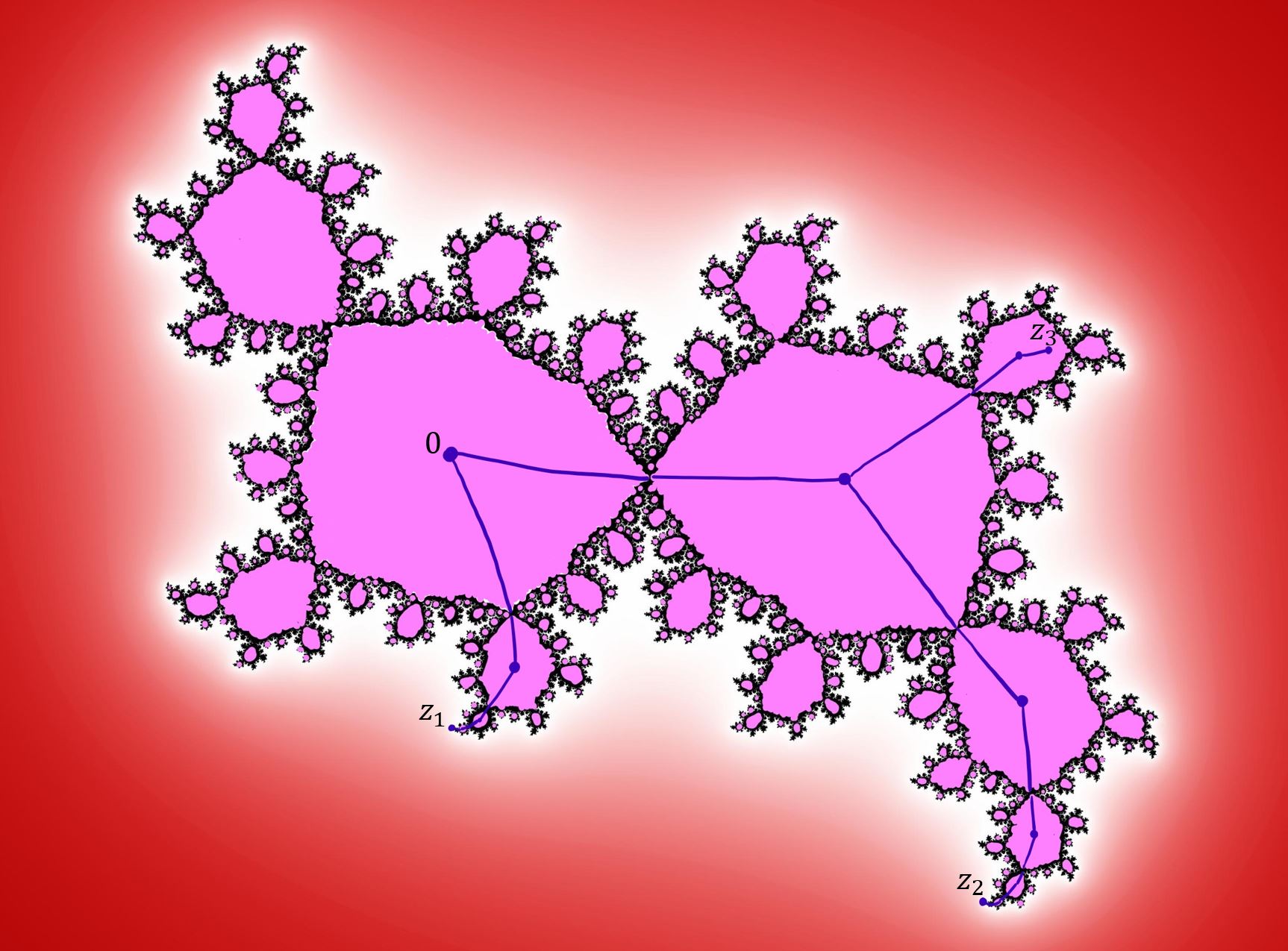}
  \caption{\small A schematic illustration of legal arcs in the dynamical plane of $Q_\la$.
  Three legal arcs are shown, connecting $0$ with points $z_1$, $z_2$, and $z_3$.
  Observe that the legal arcs from $0$ to $z_2$ and from $0$ to $z_3$ have an initial segment in common.
}\label{fig:reg}
\end{figure}

Typically, we deal with legal arcs with initial point $0$.
In the multi-angle $(\al_0,\dots,\al_k,\dots)$ of $I$ we will always have that $\al_0=\al_1, \al_2=\al_3$ etc
 because these pairs of angles correspond to pairs of internal rays of adjacent bubbles that
 eventually map onto the same internal ray of $\Delta(f)$ and, hence, have the same polar argument
(we make this more precise in Lemma \ref{l:polQ}).
Legal arcs for polynomials, under the name of regulated arcs,  were introduced by Douady and Hubbard in \cite{hubbdoua85};
 they play a key role in the definition of a \emph{Hubbard tree} for a post-critically finite polynomial.
We use legal arcs in an essentially different way.

If $z\in K(f)$ is such that there is a legal arc $I_z$ from $0$ to $z$,
then the \emph{polar multi-angle} (or just \emph{multi-angle}) of $z$ is defined
as the polar multi-angle of $I_z$. Note that if $I_z$ exists, then it is unique.

\begin{lem}
  \label{l:polQ}
For any $z\in K(Q)$, there is a legal arc $I_z$ from $0$ to $z$.
Let $(\al_0, \dots, \al_k,\dots)$ be the multi-angle of $z$.
Then each term $\al_i$, except possibly the last term, has the form $\al_i=-m_i\ta$.
Here $m_i$ are nonnegative integers such that $m_{2i+1}=m_{2i}$ and $m_{2i+2}>m_{2i+1}$.
Moreover, $z$ is uniquely determined by the multi-angle and (if the multi-angle is finite) by $\rho_z$.
\end{lem}

A sequence $(\al_0,\dots,\al_k,\dots)$ with the properties listed in Lemma \ref{l:polQ}
 is called a \emph{legal sequence of angles}.
We also define a \emph{legal angle} as an angle of the form $-m\ta$,
where $m$ is a nonnegative integer.

\begin{proof}
  Suppose first that $z$ is in the closure of a bubble $A$ of $Q$.
The argument will use induction on $\mathrm{Gen}(A)$.
If $A=\Delta(Q)$, then $I_z$ is a segment or the closure of some internal ray of $\Delta(Q)$,
 the multi-angle of $z$ is $(\ta_z)$, and $z$ is determined by $\ta_z$ and $\rho_z$.
Thus we now assume that $A\ne \Delta(Q)$.
Then $I_z$ intersects $\d \Delta(Q)$ at a point $x$ that is eventually mapped to $1$.
Let $m_0$ be the non-negative integer with $Q^{m_0}(x)=1$, then $\al_0=-m_0\ta$.
The arc $Q^{m_0}(I_z)=I_{Q^{m_0}(z)}$ contains $R_{\Delta(Q)}(0)$,
 and $Q^{m_0+1}(I_z)=I_{Q^{m_0+1}(z)}\cup\ol R_{\Delta(Q)}(\ta)$.
The multi-angle of $Q^{m_0}(z)$ starts with $0$, $0$ since
 both initial components of $I^\circ_{Q^{m_0}(z)}$ map onto $R_{\Delta(Q)}(\ta)$.
We may assume by induction that the multi-angle of $Q^{m_0+1}(z)$ is $(\tilde\al_2,\dots,\tilde\al_k)$,
 where $\tilde\al_i=-\tilde m_i\ta$ for $2\le i<k$, and $\tilde m_i$ satisfy the desired properties.
In this case the multi-angle of $z$ is $(\al_0,\dots,\al_k)$, where $\al_0=\al_1=-m_0\ta$ and
 $\al_i=\tilde\al_i-(m_0+1)\ta=-m_i\ta$ with $m_i=\tilde m_i+m_0+1$ for $i\ge 2$ and $2\le i<k$.
Thus, we proved that every point from the closure of every bubble of $Q$ has a multi-angle.
Moreover, the latter is a legal sequence of angles.

Suppose now that $z\in K(Q)$ is not in the closure of a bubble.
Then there is a sequence of pairwise different bubbles $A_0$, $\dots$, $A_k$, $\dots$ such that $A_k\to \{z\}$ in the Hausdorff metric.
Moreover, we can assume that $A_0=\Delta(Q)$ and $A_i\cap A_{i+1}=\{z_i\}$, where $z_i$ is eventually mapped to $1$.
Clearly, $I_{z_i}$ is an initial segment of $I_{z_j}$ with $j>i$.
Set $I_z$ to be the closure of the union of all $I_{z_i}$; then $I_z$ is a
legal arc from $0$ to $z$.
It follows that there is an infinite legal sequence of angles such that the multi-angle $z_i$ is an initial segment of this sequence.
This infinite legal sequence is then the multi-angle of $z$.
Thus, all points of $K(Q)$ have well-defined multi-angles.

Given a legal sequence of angles $(\al_0,\dots,\al_k,\dots)$,
 there is a unique sequence of bubbles $A_0$, $\dots$, $A_i$, $\dots$
 such that the point $z_i\in A_i\cap A_{i+1}$ has multi-angle $(\al_0,\dots,\al_{2i})$.
If the sequence $(\al_i)$ is infinite, then the corresponding sequence of bubbles converges to a unique point $z$
 determined by the infinite multi-angle $(\al_i)$.
If the sequence $(\al_i)$ is finite, then it defines a unique last bubble $A_n$ in the corresponding sequence of bubbles
 and an internal ray $R=R_{A_n}(\al_{2n})$ or $R_{A_n}(\al_{2n-1})$ in $A_n$.
All points of $R$ together with the landing point of $R$ (and no other points) have the given multi-angle.
These points are then determined by the legal sequence $(\al_0,\dots,\al_k,\dots)$ and the polar radius.
\end{proof}

Sequences of bubbles defined in the proof of Lemma \ref{l:polQ} for points $z\in K(Q)$, are called \emph{bubble rays}.

\subsection{Bubble rays and bubble chains}
\label{ss:bubblerays}
Take $P=P_c\in\Cc^c_\la$ and set $Q=Q_\la$.
Define $Y(P)$ as the set of all points $z\in K(P)$ for which there is a legal arc $I_z$ from $0$ to $z$.
Note that, by definition, $Y(P)$ includes $\ol\Delta(P)$ and is forward invariant: $P(Y(P))\subset Y(P)$.
However, in general, $Y(P)$ does not have to be closed.

Every $z\in Y(P)$ has a multi-angle and, if the latter is finite, the polar radius $\rho_z$.
Moreover, it is not hard to see that the multi-angle of $z$ is a legal sequence of angles.
Set $\rho_z=\infty$ if the multi-angle of $z$ is infinite.
Similarly, we set $\rho_w=\infty$ for points $w\in K(Q)$ not on the boundary of a bubble of $Q$.
The map $\eta_P:Y(P)\to K(Q)$ takes $z\in Y(P)$ to a unique point $w=\eta_P(z)$ with
 the same multi-angle and polar radius.
By definition of $Y(P)$ and properties of multi-angles and polar radii, $\eta_P\circ P=Q\circ \eta_P$ on $Y(P)$.

Let $A$ be a bubble of generation $n$.
If $P^n:A\to\Delta(P)$ is one-to-one, then $A$ is called \emph{off-critical}.
If $c\in A$, then $A$ is called \emph{critical}.
Finally, if $A$ is a pullback of a critical bubble, it is said to be \emph{precritical}.
For any bubble $A$, one can define its \emph{root point} $r(A)$.
When $A$ is off-critical, the root point is uniquely defined by the formula $P^{n-1}(r(A))=1$.
When $A$ is critical or precritical, there may be legal paths from $0$ to some points in $A$.
All these paths intersect the boundary of $A$ at the same point; this point is by definition the root point $r(A)$.
There are two points $z'$, $z''\in \d A$ such that $P^{n-1}(z')=P^{n-1}(z'')=1$, and the point $r(A)$ is one of them.

\begin{dfn}[Legal bubbles and bubble correspondence]
\label{d:leg}
A bubble $A$ of $P$ with $A\cap Y(P)\ne\0$ is called \emph{legal}.
Thus, $A$ is legal if and only if $r(A)\in Y(P)$, and $P^i(r(A))\ne c$ for $i<\mathrm{Gen}(A)$.
If a legal bubble $A$ is off-critical, then $\ol{A}\subset Y(P)$.
Clearly, $\eta_P(A\cap Y(P))$ lies in a unique bubble $A_Q$ of $Q$.
Say that $A$ and $A_Q$ \emph{correspond} to each other.
This correspondence between some bubbles of $P$ and 
 bubbles of $Q$ is called the \emph{bubble correspondence}.
By definition, if $A$ is a legal bubble of $P$, then $P(A)$ is also a legal bubble of $P$.
Moreover, if $A$ corresponds to $A_Q$, then $P(A)$ corresponds to $Q(A_Q)$.
\end{dfn}

Define $(\R/\Z)^\star$ as the set of nonempty finite sequences of angles and
 $(\R/\Z)^\mathbb{N}$ as the set of infinite sequences of angles.
The map
$$
\Pi:(\R/\Z)^\star\sm\{(0),(0,0)\}\to (\R/\Z)^\star
$$
acts as follows.
Take $\vec\al=(\al_0,\al_1,\al_2,\dots)\in (\R/\Z)^\star$.
If $\al_0=\al_1=0$, then $\Pi(\vec\al)=(\al_2+\ta,\dots)$,
 otherwise $\Pi(\vec\al)=(\al_0+\ta,\al_1+\ta,\al_2+\ta,\dots)$.
Then any legal sequence in $(\R/\Z)^\star$ of length $>1$ is eventually mapped to $(0)$ or $(0,0)$ under $\Pi$.
On $(0)$ and $(0,0)$, the map $\Pi$ is undefined.
Clearly, $\Pi$ can be also defined as a self-map of $(\R/\Z)^\mathbb{N}$, by the same rule.

Let $A$ be a legal bubble of $P$ of generation $n$.
If $A$ is off-critical, recall that the center of $A$ is by definition the preimage of $0$ under $P^n:A\to\Delta(P)$.
The \emph{incoming} radius of any off-critical bubble $A$ is its radius $R(A)$ connecting $r(A)$ to its center.
Now, if $c\in A$ then $P|_A$ is two-to-one.
To define the center of $A$, recall the earlier analysis of pullbacks of radii into critical bubbles.
It follows from that analysis that there are two cases depending on the mutual location of $R(P(A))$ and $P(c)$.
If $R(P(A))$ does not contain $P(c)$, or $P(c)$ is the center of $P(A)$, then
there is a unique pullback of $R(P(A))$
that connects $r(A)$ with a preimage of the center of $P(A)$, and this preimage
of the center of $P(A)$ is said to be the \emph{center} of $A$. The remaining case is
when $R(P(A))$ contains $P(c)$ but $P(c)$ is not the center of $P(A)$. In that case
the center of $A$ is not defined.

Finally, let $A$ be precritical; then $P|_A$ is one-to-one.
If the center of $P(A)$ is defined, set the center of $A$ to be the pullback of the
center of $P(A)$ into $A$; if the center of $P(A)$ is not defined, then the center of $A$ is not defined either.
If the center of $A$ is defined, it is denoted $o_A$.

The \emph{multi-angle of $A$} is defined as the multi-angle of the root point of $A$.
If $A$ has multi-angle $\vec\al$, then $P(A)$ has multi-angle $\Pi(\vec\al)$.
We can now describe multi-angles of legal bubbles.

\begin{prop}
  \label{p:ma-bub}
  Let $\vec\al$ be a finite legal sequence of angles of odd length starting with $-m\ta$ for a nonnegative $m\in\Z$.
Then $\vec\al$ is a multi-angle of some legal bubble if and only if
 no eventual $\Pi$-image of an initial subsequence of $\vec\al$ is the multi-angle of $c$.
\end{prop}

The assumption that $\vec\al$ starts with $-m\ta$ is essential if $\vec\al$ has length 1
(otherwise it follows from the definition of a legal sequence).

\begin{proof}
By Definition \ref{d:leg}, a bubble $A$ is legal if and only if no image $P^i(I_{r(A)})$ with $0\le i<\mathrm{Gen}(A)$
 contains $c$.
On the other hand, $c\in P^i(I_{r(A)})$ if and only if the $\Pi^i$-image of an initial subsequence
of $\vec\al$ is the multi-angle of $c$.
\end{proof}

The concept of a bubble ray was used in the proof of Lemma \ref{l:polQ}.

\begin{dfn}[Bubble rays, bubble chains, core curves]
\label{d:bubray}
Take a legal bubble $A$ of $P$ and a point $z\in \ol{A}\cap Y(P), z\ne r(A)$ .
A legal arc $I_z$ from $0$ to $z$ passes through bubbles $A_0=\Delta(P)$, $\dots$, $A_n=A$
in this order and through no other bubbles.
The sequence $A_0$, $\dots$, $A_n$ is called a \emph{bubble chain (to $z$)}.
A \emph{bubble ray} is a sequence $\Ac=(A_0, A_1, \dots)$ of legal bubbles $A_i$ such that
$A_0,$ $\dots,$ $A_n$ is a bubble chain, for every finite $n$. Set $\bigcup \Ac=\bigcup_{i\ge 0}A_i$.
Bubble chains and bubble rays for $Q=Q_\la$ are defined similarly.
The \emph{core curve} of $\Ac$ is defined as the union of $I_{z_i}$, where $z_i\in \ol A_i\cap\ol A_{i+1}$.
(Note that $I_{z_i}\subset I_{z_j}$ for $i<j$.)
\end{dfn}

\begin{dfn}[Landing bubble rays]
  Consider a bubble ray $\Ac=(A_i)$ for $P$.
We say that $\Ac$ \emph{lands} at a point $z$ if $\{z\}$ is the upper limit of the sequence $A_i$, that is
$$
\{z\}=\bigcap_i \ol{A_i\cup A_{i+1}\cup \dots}.
$$
In general, the right hand side is denoted by $\lim\Ac$ and is called the \emph{limit set} of $\Ac$.
If $\Ac$ lands at $z$, then we also say that $\Ac$ is a bubble ray \emph{to $z$}.
Similar definitions apply to the dynamical plane of $Q$.
\end{dfn}

Consider a bubble ray $\Ac=(A_i)$ for $P$.
If $P(A_1)\ne A_0$, then we define $P(\Ac)$ as $(A_0$, $P(A_1)$, $P(A_2)$, $\dots)$.
Otherwise, $P(A_1)=A_0$, and we define $P(\Ac)$ as $(A_0$, $P(A_2)$, $P(A_3)$, $\dots)$.
If $P^m\left(\bigcup\Ac\right)=\bigcup\Ac$, then $\Ac$ is said to be \emph{periodic of period $m$}.
Let $I$ be the core curve of $\Ac$.
If $m$ is the minimal period of $\Ac$ under $P$, then $P^m(I)=I$.
However, $P^m:I\to I$ is not one-to-one; $I$ is folded at critical points of $P^m$.

\begin{lem}
  \label{l:ubub}
For $z$ in the dynamical plane of $P$, there is at most one bubble chain or a bubble ray to $z$.
\end{lem}

\begin{proof}
Suppose that $\Ac'$ and $\Ac''$ are different bubble rays or bubble chains to $z$.
If $\Ac'$ is a bubble ray, then set $I'$ to be its core curve; otherwise set $I'=I_z$.
The arc $I''$ is defined similarly, with $\Ac'$ replaced by $\Ac''$.
If $\Ac'\ne\Ac''$, then there is a bounded open set $U$ in $\C$ whose boundary is contained in $I'\cup I''$.
By the Maximum Modulus Principle, the sequence $P^n$ is bounded on $U$, hence equicontinuous.
We conclude that $U$ is in the Fatou set, that is, $U$ is in a single bubble, a contradiction.
\end{proof}

\subsection{Landing of bubble rays}
\label{ss:landbubrays}
Recall that $P\in \Cc^c_\la$.
If $\Ac=(A_n)$ is a periodic bubble ray for $P$ of minimal period $m$, then,
clearly, $P^m$ takes several first bubbles $A_0$, $\dots$, $A_{k}$ to $A_0$, and $A_{k+1}$ to $A_1$.
In this case we say that \emph{$P^m$ shifts bubbles of $\Ac$ by $k$}.
We always have $k\ge 1$.

The main result of this subsection is Theorem \ref{t:land}.
Theorem \ref{t:752} will be used in the proof of Theorem \ref{t:land}.
If $X\subset\C$ is a continuum, $\thu(X)$ stands for its \emph{topological hull}, i.e., the union of $X$ and
 all bounded complementary components of $X$.

\begin{thm}[Theorem 7.5.2 \cite{bfmot13}]\label{t:752}
Let $f$ be a polynomial, let $K(f)$ be connected, and let $X\subset J(f)$ be an invariant continuum.
Suppose that $X$ is not a singleton.
Then $\thu(X)$ contains a rotational fixed point or an invariant parabolic domain.
\end{thm}

Here a \emph{rotational fixed point} means one of the following:
\begin{itemize}
\item an attracting fixed point;
  \item a repelling or parabolic fixed point where no invariant external ray lands;
  \item a Siegel point;
  \item a Cremer point.
\end{itemize}
In other words, a fixed point is rotational unless it is the landing point of
 some invariant external ray.

Theorem \ref{t:752} is related to the following result of \cite{GM}.
Let $f$ be a polynomial of any degree $>1$.
Consider the union $\Sigma_f$ of all invariant external $f$-rays with the set $\mathrm{Fix}_f$ of their landing points.
In other words, $\mathrm{Fix}_f$ is the set of all repelling or parabolic fixed points of $f$.
A \emph{rotational object} of $f$ is defined as either a rotational fixed point or an invariant parabolic domain.

\begin{thm}[\cite{GM}]
\label{t:gm}
Every component of $\C\sm\Sigma_f$ contains a unique invariant rotational object of $f$.
\end{thm}

A subset of $\Sigma_f$ consisting of two rays landing at the same point and their common landing point is called a \emph{cut}.
Theorem \ref{t:gm} can be restated as follows:
 any pair of different invariant rotational objects for $f$ is separated by a cut from $\Sigma_f$.

\begin{thm}
  \label{t:land}
Let $\Ac$ be a periodic bubble ray for $P$.
Then $\Ac$ lands at a periodic repelling or parabolic non-rotational point of $P$.
\end{thm}

\begin{proof}
Let $L$ be the limit set of $\Ac$. It is easy to see that $L$ and $\Delta(P)$ are disjoint as otherwise
some boundary points of some bubbles from $\Ac$ will be shielded from infinity, a contradiction.

Let $L$ be of minimal period $m$, and consider the map $f=P^m$.
It suffices to prove that $L$ is a singleton.
Suppose otherwise.
Then by Theorem \ref{t:752}, the set $L$ contains an $f$-invariant rotational
object $T$ (rotational $f$-fixed point or an $f$-invariant parabolic domain).
As above, construct the set $\Sigma_f$; by Theorem \ref{t:gm}, one of its cuts separates $T$ and $0$.
Evidently, $\Ac$ cannot intersect this cut
which implies that $L$ must be located on one side of the cut while $T$ is located on the other side. A contradiction.
Hence $L$ is an $f$-fixed point. Since it belongs to $J(P)$, it is not attracting. If it is Cremer or Siegel, then,
again relying on Theorem \ref{t:gm}, we separate $L$ from $0$ with a rational cut, again a contradiction. Hence
$L$ is an $f$-fixed repelling or parabolic point $a$.
If $a$ is rotational,
then periodic rays landing at $a$ undergo a nontrivial combinatorial rotation under $f$
(since, by definition of a rotational fixed point, no invariant external ray can land at $a$).
Let $W\supset \bigcup\Ac$ be a wedge bounded by two consecutive $f$-rays landing at $a$.
Locally near $a$, the wedge $W$ is mapped to some other wedge disjoint from $W$.
A contradiction with $\bigcup\Ac\subset W$.
\end{proof}

\section{Stability}
\label{s:stab}
We start with a very general continuity property.
Let $\mathrm{Rat}_d$ be the space of all degree $d$ rational self-maps of $\ol\C$ with the topology of uniform convergence.
We also write $\mathrm{Comp}$ for the space of all compact subsets of $\ol\C$ with the Hausdorff metric.
Note that the Hausdorff metric on $\mathrm{Comp}$ as well as the uniform convergence on $\mathrm{Rat}_d$
 are associated with the spherical metric on $\ol\C$.
The following lemma is basically a consequence of the Open Mapping property of holomorphic functions.

\begin{lem}
  \label{l:Hcont}
  Consider the map from $\mathrm{Rat}_d\times\mathrm{Comp}\to\mathrm{Comp}$ given by
  $$
  (f,X)\mapsto f^{-1}(X).
  $$
  This map is continuous.
\end{lem}

\begin{proof}
In what follows, ``$\eps$-close'' means ``at distance at most $\eps$''.
Fix $(f,X)\in\mathrm{Rat}_d\times\mathrm{Comp}$.
Choose $\eps>0$.
We need to show that, if $\delta=\delta(\eps)>0$ is sufficiently small
and $(g,Y)$ is $\delta$-close to $(f,X)$, then $g^{-1}(Y)$ is $\eps$-close to $f^{-1}(X)$.
Here $(g,Y)$ being $\delta$-close to $(f,X)$ means that $g$ is $\delta$-close to $f$ and $Y$ is $\delta$-close to $X$.
By definition, $g^{-1}(Y)$ being $\eps$-close to $f^{-1}(X)$ means
 that for every point $x\in f^{-1}(X)$, there is $y\in g^{-1}(Y)$ that is $\eps$-close to $x$, and vice versa:
 for every $y$ with $g(y)\in Y$, there is $x\in f^{-1}(X)$ that is $\eps$-close to $y$.

First, take $x\in f^{-1}(X)$.
Then $g(x)$ is $\delta$-close to $f(x)$.
There is a point $y^*\in Y$ that is $\delta$-close to $f(x)$, since $Y$ is $\delta$-close to $X$.
Finally, $y^*$ being $2\delta$-close to $g(x)$ implies the existence of $y\in g^{-1}(y^*)$ that is $\eps$-close to $x$.
Moreover, the corresponding choice of $\delta$ can be made independent of $g$.
Indeed, by the Open Mapping property, the $f$-image of the $\eps$-neighborhood of $x$
 includes the $4\delta$-neighborhood of $f(x)$.
Hence, the $g$-image of the $\eps$-neighborhood of $x$ incudes the $3\delta$-neighborhood of $f(x)$,
 and the latter includes the $2\delta$-neighborhood of $g(x)$.

Now take $y\in g^{-1}(Y)$; the argument is similar to the above.
By the assumption, $g(y)$ is $\delta$-close to $f(y)$, and there is a point $x^*\in X$ that is $\delta$-close to $g(y)$.
Since $x^*$ is $2\delta$-close to $f(y)$, there is a point $x$ such that $f(x)=x^*$ and $x$ is $\eps$-close to $y$.
\end{proof}

We will need the following corollary of Lemma \ref{l:Hcont}.

\begin{cor}
\label{c:comp-cont}
Suppose that $(f,X)\in\mathrm{Rat}_d\times\mathrm{Comp}$, that $X$ is connected, and that $A$ is a component of $f^{-1}(X)$.
If there are no critical points of $f$ in $A$, then for all $(g,Y)$ close to $(f,X)$ there is
 a component of $g^{-1}(Y)$ close to $A$ and not containing critical points of $g$.
\end{cor}

\begin{proof}
For some $\eps>0$, the $5\eps$-neighborhood of $A$ maps homeomorphically onto a neighborhood of $X$.
Take $(g,Y)$ so close to $(f,X)$ that $g^{-1}(Y)$ is $\eps$-close to $f^{-1}(X)$.
This is possible by Lemma \ref{l:Hcont}.
Moreover, we may assume that $g$ is injective on the $4\eps$-neighborhood of $A$.
Let $B$ be a component of $g^{-1}(Y)$ intersecting the $2\eps$-neighborhood of $A$.
It follows that $B$ lies entirely in the $2\eps$-neighborhood of $A$.
Indeed, if a point of $B$ is at distance $2\eps$ from $A$, then it cannot be $\eps$-close to $f^{-1}(X)$
 by the assumption that $f$ is injective on the $5\eps$-neighborhood of $A$.
In particular, the closest point to any $b\in B$ is in $A$, and this closest point is $\eps$-close to $b$.
Observe also that $B$ is the only component of $g^{-1}(Y)$ in the $2\eps$-neighborhood of $A$, since
 the restriction of $g$ to the latter neighborhood is injective.
Now take any $a\in A$, and let $b$ be the closest to $a$ point of $g^{-1}(Y)$.
Then $b$ is $\eps$-close to $a$, hence $b\in B$.
We see that $A$ and $B$ are $\eps$-close, as desired.
\end{proof}

\subsection{Equicontinuous motion}
Let $A\subset \ol\C$ be any subset and $\Lambda$ be a metric space with a marked base point $\tau_0$.
A map $(\tau,z)\mapsto \iota_\tau(z)$ from $\Lambda\times A$ to $\ol\C$ is an \emph{equicontinuous motion
 (of $A$ over $\Lambda$)} if $\iota_{\tau_0}=id_A$, the family of maps $\tau\mapsto \iota_\tau(z)$ parameterized by $z\in A$ is equicontinuous,
 and $\iota_\tau$ is injective for every $\tau\in\Lambda$.
An equicontinuous motion is \emph{holomorphic} if $\Lambda$ is a Riemann surface,
 and each function $\tau\mapsto \iota_\tau(z)$, where $z\in A$, is holomorphic.
By the \emph{$\la$-lemma} of \cite{MSS}, to define a holomorphic motion, it is enough to require that
 every map $\iota_\tau$ is injective, and $\iota_\tau(z)$ depends holomorphically on $\tau$, for every fixed $z$.
Then the family of maps $\iota_\tau$ is automatically equicontinuous.
Suppose now that $F_\tau:\ol\C\to\ol\C$ is a family of rational maps such that $F_{\tau_0}(A)\subset A$.
An equicontinuous motion $(\tau,z)\mapsto \iota_\tau(z)$ is \emph{equivariant} with respect to the family $F_\tau$ if $\iota_\tau(F_{\tau_0}(z))=F_\tau(\iota_\tau(z))$ for all $z\in A$.
If the family $F_\tau$ is clear from the context, then we also say that the holomorphic motion
 \emph{commutes with the dynamics}.

We first study the equicontinuous motion of the Siegel disk $\Delta(P)$.
The following theorem is an easy consequence of known results.

\begin{thm}
\label{t:sul}
Choose an arbitrary base point $P_0\in \Cc^c_\la$ and an arbitrary point $z\in\ol\Delta(P_0)$.
There is an equivariant equicontinuous motion $\iota_P$ of $\ol \Delta(P)$ over $\Cc^c_\la$ such that $\iota_{P_0}=id$.
Moreover, $\iota_P(z)$ has the same polar coordinates in $\ol \Delta(P)$ as $z$ in $\ol \Delta(P_0)$.
If $P_0\notin\Zc^c_\la$, then this equicontinuous motion is holomorphic on $\Cc^c_\la\sm\Zc^c_\la$.
\end{thm}

\begin{proof}
Set $\iota_P(z)=\ol\psi_{\Delta(P)}\circ\ol\psi_{\Delta(P_0)}^{-1}$.
By Corollary \ref{c:sup-conv}, the function $\iota_P$ depends continuously on $P$ with respect to the sup-norm.
This means that $\iota_P$ is an equicontinuous motion.
The equivariance follows from the fact that $\ol\psi_{\Delta(P)}$ conjugates the rotation by $\ta$ with $P|_{\ol\Delta(P)}$.

Suppose now that $P_0\notin\Zc^c_\la$ and $P$ runs through $\Cc^c_\la\sm\Zc^c_\la$.
The $P$-orbit of $1$ moves holomorphically with $P\in\Cc^c_\la\sm\Zc^c_\la$;
By \cite{MSS}, this motion extends to an equivariant holomorphic motion of $\d \Delta(P)$, cf. \cite{che20}.
By a remark of D. Sullivan \cite{sul} (see \cite{zak16} for a published proof), there exists an equivariant holomorphic motion
 $\iota_P:\ol \Delta(P_0)\to \ol \Delta(P)$ that extends the holomorphic motion of $\d \Delta(P)$
 and is such that $\iota_P:\Delta(P_0)\to \Delta(P)$ is a conformal isomorphism taking $0$ to $0$ and $1$ to $1$.
By the uniqueness of the Riemann map the map $\iota_P$ is the same as before.
In particular, $\iota_P$ preserves the polar coordinates.
\end{proof}

\subsection{Stability of legal arcs}
The equicontinuous motion of $\ol\Delta(P)$ extends to some other dynamically defined subsets.

\begin{dfn}[Stability]
  Consider $P_0\in\Cc^c_\la$ and a subset $A\subset Y(P_0)$.
Since $Y(P_0)$ is by definition forward invariant, it follows that $P_0^n(A)\subset Y(P_0)$ for all $n\ge 0$.
Set $B=\bigcup_{n\ge 0} P_0^n(A)$.
Say that $A$ is \emph{stable} (or \emph{$\la$-stable}) if there is an equivariant equicontinuous motion $\{\iota_P^B\}$
 of $B$ over an open neighborhood of $P_0$ in $\Cc^c_\la$ such that, for every $z\in B$,
 the point $z\<P\>=\iota^B_P(z)$ has the same multi-angle and the same polar radius as $z$.
Clearly, if such an equicontinuous motion exists, then it is unique.
Write $A\< P\>$ for $\iota^B_P(A)$ etc.
If the equicontinuous motion is in fact holomorphic, then say that $A$ is \emph{holomorphically stable}.
\end{dfn}

\begin{lem}
  \label{l:regarc-mov}
  Take $P_0\in\Cc^c_\la$ and a point $z\in Y(P_0)$ that has a finite multi-angle.
If $z$ is never mapped to $c$, or if $c\in\ol{\Delta(P_0)}$,
 then the legal arc $I_z$ from $0$ to $z$ in $K(P_0)$ is stable.
It is holomorphically stable if $P_0\notin\Zc^c_\la$.
\end{lem}

\begin{proof}
  Suppose that $\vec\al=(\al_0,\dots,\al_k)$ is the multi-angle of $z$.
If $k=0$, then the statement follows from Theorem \ref{t:sul}.
If $\vec\alpha=(0,0)$, then the statement follows from Theorem \ref{t:sul} and Lemma \ref{l:Hcont}.
Thus we assume that $k>0$ and $\vec\alpha\ne (0,0)$.
Since $\vec\al$ is a legal sequence of angles, $\Pi^m(\vec\al)=(0)$ or $(0,0)$ for some minimal integer $m>0$.
The subsequent argument employs induction on both $m$ and $k$.
Let $w$ be the last (closest to $z$) point of $I_z$ with multi-angle $(\al_0,\dots,\al_{k-1})$.
Then $I_z$ is the concatenation of $I_w$ (=the legal arc from $0$ to $w$), and $I_{[w,z]}$ (=the legal arc from $w$ to $z$).
By induction on $k$, assume that $I_w$ is stable.
In particular, $w\< P\>$ is defined for all $P$ close to $P_0$, and $w\< P\>$ has the same multi-angle and polar radius as $w$.

By induction on $m$, assume that $P_0(I_z)$, hence also $T=P_0(I_{[w,z]})$, are stable.
Thus $T\< P\>$ depends continuously on $P$ in the Hausdorff metric.
Define $I_{[w,z]}\< P\>$ as the $P$-pullback of $T\<P\>$ containing $w\<P\>$.
Note that $T\<P\>$ is free from critical values of $P$ by our assumptions.
If $P$ is close to $P_0$, then, by Corollary \ref{c:comp-cont}, the set $I_{[w,z]}\<P\>$ is close to $I_{[w,z]}$.
In particular, $I_{[w,z]}\<P\>$ connects $w\<P\>$ with a point $z\<P\>$ that is close to $z$.
Moreover, $I_{[w,z]}\<P\>$ contains no critical points of $P$ and maps forward by $P$ in a homeomorphic fashion.
It follows that a suitable inverse branch of $P$ on $T\<P\>$ defines an equicontinuous motion of $I_{[w,z]}$.

Thus both $I_w$ and $I_{[w,z]}$ are stable.
It follows that their concatenation $I_z$ is also stable, as desired.
If $P_0\notin\Zc^c_\la$, then the argument goes through with ``equicontinuous'' replaced by ``holomorphic''.
\end{proof}

We now discuss stability of infinite periodic legal arcs.

\begin{thm}
  \label{t:stab}
  Let $\Ac$ be a periodic bubble ray for $P_0\in\Cc^c_\la$ landing at a repelling periodic point $x$.
Suppose that $P_0^i(I_x)$ does not contain $c$ for $i\ge 0$.
Then $I_x$ is stable; it is holomorphically stable if $P_0\notin\Zc^c_\la$.
\end{thm}

Note that, under assumptions of Theorem \ref{t:stab}, the arc $I_x$ is the core curve of $\Ac$.

\begin{proof}
  Let $m$ be the minimal $P_0$-period of $\Ac$, then $P_0^m(x)=x$.
Since $x$ is repelling, there is a small round disk $D$ around $x$ such that $D\Subset P_0^m(D)$,
and the map $P_0^m:D\to P_0^m(D)$ is a homeomorphism.
Suppose that $\Ac=(A_n)$.
Since $\Ac$ lands at $x$, then $\ol A_n\subset D$ for all $n\ge N$ for some $N$.
Also, $P_0^m$ shifts bubbles of $\Ac$ by a certain integer $k\ge 1$.

Clearly, there is a point $y\in I_x$ such that
$$
I_y=I_x\cap \ol{A_0\cup\dots\cup A_{N+k}}.
$$
By Lemma \ref{l:regarc-mov}, the legal arc $I_y$ is stable.
In particular, for $P$ close to $P_0$, there is a legal arc $I_y\<P\>$ close to $I_y$ and with the same multi-angle.
Moreover, $I_y\<P\>$ passes through legal bubbles $A_0\<P\>$, $\dots$, $A_{N+k}\<P\>$ of $P$ and terminates at $y\<P\>$.
Consider the point $z=P_0^m(y)\in I_y$ and the segment $I_{[z,y]}$ of $I_x$ from $z$ to $y$.
Then $I_{[z,y]}$ is also stable, the corresponding segment $I_{[z,y]}\<P\>$ for $P$ connects $z\<P\>$ with $y\<P\>$.
Note also that $I_{[z,w]}\subset D$ (the point $z$ belongs to the closure of $A_N$).

If $P$ is close to $P_0$, then $D\Subset P^m(D)$, and $P^m:D\to P^m(D)$ is a homeomorphism.
Write $P^{-m}_D$ for the inverse of this homeomorphism.
Then $P^{-m}_D$ is a well-defined holomorphic map on $D$ depending analytically on $P$.
Since $x$ is repelling, it is stable, so that there is a nearby repelling point $x\<P\>$ for $P$ of the same period.
Set
$$
I_x\<P\>=I_y\<P\>\cup
\left(\bigcup_{k=1}^\infty (P^{-m}_D)^k(I_{[z,y]}\<P\>)\right)
\cup\{x\<P\>\}.
$$
Every term in the right-hand side moves equicontinuously with $P$ as long as $P$ stays close to $P_0$.
The infinite union moves equicontinuously since for $P^{-m}_D$ the point $x\<P\>$ is attracting
 (the iterates cannot inflate the modulus of continuity).
It is also clear that the motion is holomorphic provided that $P_0\notin\Zc^c_\la$ and $P$ is close to $P_0$.
\end{proof}

\subsection{Stability of Siegel rays}
Theorem \ref{t:stab} parallels a classical result on stability of periodic external rays landing at repelling points.

\begin{lem}[\cite{hubbdoua85}, cf. Lemma B.1 \cite{GM}]
 \label{l:rep}
Let $P_0$ be a polynomial, and $z$ be a repelling periodic point of $P_0$.
If an external ray $R_{P_0}(\theta)$ with rational argument $\theta$ lands at $z$, then,
 for every polynomial $P$ sufficiently close to $P_0$,
 the ray $R_{P}(\theta)$ lands at a repelling periodic point $z\<P\>$ of $P$ close to $z$,
and $z\<P\>$ depends holomorphically on $P$.
\end{lem}

Consider a periodic bubble ray $\Ac$ for $P_0$ and its core curve $I$.
By Theorem \ref{t:land}, the bubble ray $\Ac$ lands at a repelling or parabolic point $a$.
Let $m$ be the minimal period of $\Ac$, then $P^m(a)=a$.
Clearly, $I$ also lands at $a$,
and it is easy to see that $I=I_a$ is a legal arc from $0$ to $a$.

\begin{dfn}[Siegel rays]
Let $I$ and $a$ be as above.
By the classical Landing Theorem for polynomials (see e.g. \cite[Theorem 18.11]{mil06}),
one or several periodic external rays for $P$ land at $a$.
Let $R$ be an external ray landing at $a$.
Then $I\cup\{a\}\cup R$ is a simple curve connecting $0$ with $\infty$.
It is called a \emph{Siegel ray}.
The \emph{argument} of the Siegel ray $I\cup\{a\}\cup R$ is defined as the argument of $R$.
\end{dfn}

The following are immediate properties of Siegel rays.
Every Siegel ray originates at $0$ and extends to $\infty$.
Every Siegel ray contains precisely one periodic point $a\ne 0$; this point $a$ is repelling or parabolic.
Two different Siegel rays may have some initial segment in common.
They branch off either at an iterated preimage of $0$ or at a landing point of some bubble ray.
An external ray for $P_0$ is either disjoint from a Siegel ray or lies in the Siegel ray.

Theorem \ref{t:S-stab} below follows from Theorem \ref{t:stab} and Lemma \ref{l:rep}.

\begin{thm}
  \label{t:S-stab}
Let $\Sigma$ be a Siegel ray for $P_0\in\Cc^c_\la$.
Suppose that the non-zero periodic point in $\Sigma$ is repelling.
Then, for all $P\in\Cc^c_\la$ sufficiently close to $P_0$, there is a Siegel ray $\Sigma\<P\>$
 close to $\Sigma$ in the spherical metric and having the same argument.
Moreover, the periodic point in $\Sigma\<P\>$ depends holomorphically on $P$ provided that $P\notin\Zc^c_\la$.
\end{thm}

The only reason we require that $P\notin\Zc^c_\la$ in Theorem \ref{t:S-stab}
 is that holomorphic functions are defined on Riemann surfaces,
 and $\Cc^c_\la\sm\Zc^c_\la$ rather than $\Cc^c_\la$ has a natural structure of a Riemann surface.

\subsection{Siegel wedges}
Let $\Sigma$ and $\Sigma'$ be two Siegel rays for $P$.
By definition, they originate at $0$ and extend all the way to infinity.
Let $b$ be the point where $\Sigma$ and $\Sigma'$ branch off.
Assume that $b$ is an iterated preimage of $0$ rather than a periodic repelling or parabolic point.
Consider a wedge $W$ bounded by segments of $\Sigma$ and $\Sigma'$ from $b$ to infinity.
Notice that there are two such wedges; either wedge is called
a \emph{Siegel wedge (bounded by Siegel rays $\Sigma$ and $\Sigma'$)}, see Fig. \ref{fig:W}.
We also say that $b$ is the \emph{root point} of $W$. Set $I_b=\Sigma\cap\Sigma'$;
of the two wedges bounded by $\Sigma$ and $\Sigma'$ one contains $I_b$ and the other one is disjoint from $I_b$.
If $W$ is a Siegel wedge bounded by $\Sigma$ and $\Sigma'$ and
disjoint from $I_b$, call the multi-angle of $b$ the \emph{multi-angle of $W$}.
Otherwise (i.e. if $W$ contains $I_b\sm\{b\}$), we set the multi-angle of $W$ to be $()$ (the empty sequence).
If $\Sigma\cap\Sigma'=\{0\}$, we set the multi-angle of $W$ to be $()$ too.
The following property of Siegel wedges is immediate from the definitions.

\begin{figure}
  \centering
  \includegraphics[width=8cm]{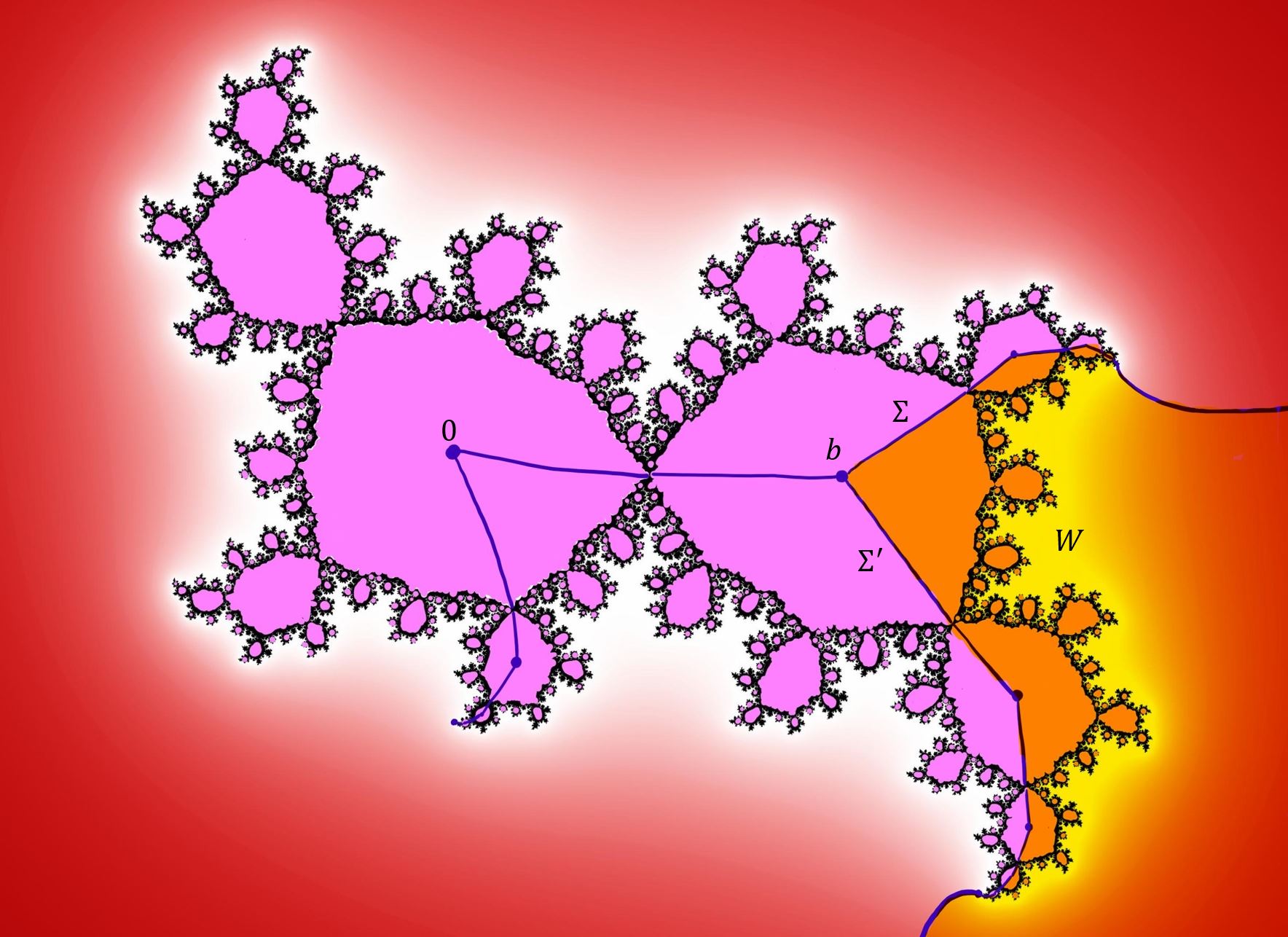}
  \caption{\small A Siegel wedge $W$ in the dynamical plane of $Q$.
  Here, $W$ is bounded by Siegel rays $\Sigma$ and $\Sigma'$ that
  have the legal arc $I_b$ from 0 to $b$ as the common initial segment
  and that branch off at point $b$.}
  \label{fig:W}
\end{figure}

\begin{prop}
  \label{p:SW-ma}
  Let $W$ be a Siegel wedge of multi-angle $\vec\al$.
Then the multi-angles of all points in $W\cap Y(P)$ contain $\vec\al$ as an initial segment.
\end{prop}

Fix a Siegel wedge $W$.
Recall that $\d W\cap K(P)\subset Y(P)$; the map $\eta_P:Y(P)\to K(Q)$ takes $z\in Y(P)$ to a unique point $w=\eta_P(z)$ with
the same multi-angle and polar radius.
Then $\eta_P(\d W\cap K(P))$ is the union of the core curves of two periodic bubble rays for $Q$.
These core curves land at some repelling periodic points, say, $x$ and $y$ of $Q$ (these are endpoints of $K(Q)$
as $K(Q)$ has no periodic cutpoints).
There are unique external rays landing at $x$ and $y$.
The union $\Gamma_Q$ of these external rays and $\eta_P(\d W\cap K(P))$ bounds a unique Siegel wedge $W_Q$
of $Q$ that contains points of $\eta_P(W\cap Y(P))$.
The wedge $W_Q$ is said to \emph{correspond} to the Siegel wedge $W$ of $P$.
Observe that since the endpoints of $\d W\cap K(P)$ may be cutpoints of $K(P)$,
 there may be several Siegel wedges of $P$ corresponding to the same Siegel wedge of $Q$.

\section{The dynamical map $\eta_P$}
\label{s:dynmap}
We now define a $P$-invariant continuum $X(P)\supset \ol{Y(P)}$ and
 extend the map $\eta_P:Y(P)\to K(Q)$ to $X(P)$.
If $\ol{Y(P)}$ contains no parabolic points of $P$, then we set $X(P)=\ol{Y(P)}$.
Suppose now that there is a parabolic periodic cycle in $\ol{Y(P)}$; let $a$ be a point in this cycle.
By the Fatou--Shishikura inequality, the cycle of $a$ is the only parabolic cycle of $P$.
In this case, let $X(P)$ be the union of $\ol{Y(P)}$ and the closures of all 
 immediate parabolic basins associated with the cycle of $a$.
Clearly, $X(P)$ is a forward invariant continuum.

\subsection{The structure of $X(P)$}
\label{ss:XP}
Consider possible intersections of $X(P)$ with bubbles of $P$.

\begin{lem}
  \label{l:XPbub}
Let $A$ be a bubble of $P$.
Suppose that a point $z\in\ol A\cap X(P)$ is different from the root point $r(A)$ of $A$.
Then $A$ is a legal bubble, and the entire bubble chain to $z$ consists of legal bubbles.
\end{lem}

\begin{proof}
Since $z\ne r(A)$, then $r(A)\in Y_P$ (otherwise points like $z$ would not exist) and
the legal arc $I_z\subset K_P$ from $0$ to $z$ is non-disjoint from $A$;
hence $A\cap X(P)\ne\0$. Since $A$ is open, $A\cap Y(P)\ne\0$, that is, $A$ is legal.
Also, $I_z$ intersects every bubble in the bubble chain to $z$, it follows that all bubbles in this chain are legal.
\end{proof}

The following is an immediate corollary of Lemma \ref{l:XPbub}.

\begin{cor}
  \label{c:XPbub}
  Let $A$ be a bubble of $P$.
Either $A$ is legal, or $\ol A$ has no points of $X(P)$ except possibly $r(A)$ in which
case $r(A)$ is eventually mapped to $c$, non-strictly before it is mapped to $1$ and
 strictly before $A$ is mapped to $\Delta(P)$.
\end{cor}

Suppose now that $A$ is legal but $A\not\subset Y(P)$.
Then there is a point $z\in A$ that is eventually mapped to $c$, say $P^n(z)=c$.
The intersection $A\cap Y(P)$ is then a proper subset of $A$ whose geometry
 is described below in Theorem \ref{t:XPinbub}.
The radial vector field $\d/\d\rho$ is well-defined on $P^{n+1}(A)\sm\{o\}$,
 where $o$ is the center of $P^{n+1}(A)$.
The integral curves of $\d/\d\rho$ are precisely the internal rays of $P^{n+1}(A)$.
The pullback of $\d/\d\rho$ to $A$ is also a well-defined vector field $v$ on $A\sm P^{-(n+1)}(o)$.
However, $v$ has zero at $z$, and there are two special integral curves of $v$
 whose $\alpha$-limit set coincides with $\{z\}$ and whose $\omega$-limit sets are in $\d A$.
These integral curves are called \emph{separatrices}.

\begin{thm}
\label{t:XPinbub}
Let $A$ be a legal bubble of $P$ such that a point $z\in A$ is eventually mapped to $c$,
 and $A_Q$ be the corresponding bubble of $Q$.
Then $\eta_P:Y(P)\cap A\to A_Q$ extends to a continuous map $\eta_P:X(P)\cap A\to A_Q$,
and one of the following two cases holds:
\begin{enumerate}
  \item The set $X(P)\cap A$ is
  the closure of a separatrix.
  The map $\eta_P:X(P)\cap A\to A_Q$ is one-to-one, and $\eta_P(X(P)\cap A)$ is a terminal segment of
  the internal ray of $A_Q$ landing at the root point of $A_Q$.
  \item The set $X(P)\cap A$ is a sector of $A$ bounded by the two separatrices together with $z$.
  It is mapped under $\eta_P$ onto $A_Q$, the boundary of the sector mapping two-to-one,
  and otherwise the map being one-to-one.
\end{enumerate}
\end{thm}

Cases (1) and (2) of Theorem \ref{t:XPinbub} are illustrated in Fig. \ref{fig:53}.

\begin{figure}
  \centering
  \includegraphics[width=8cm]{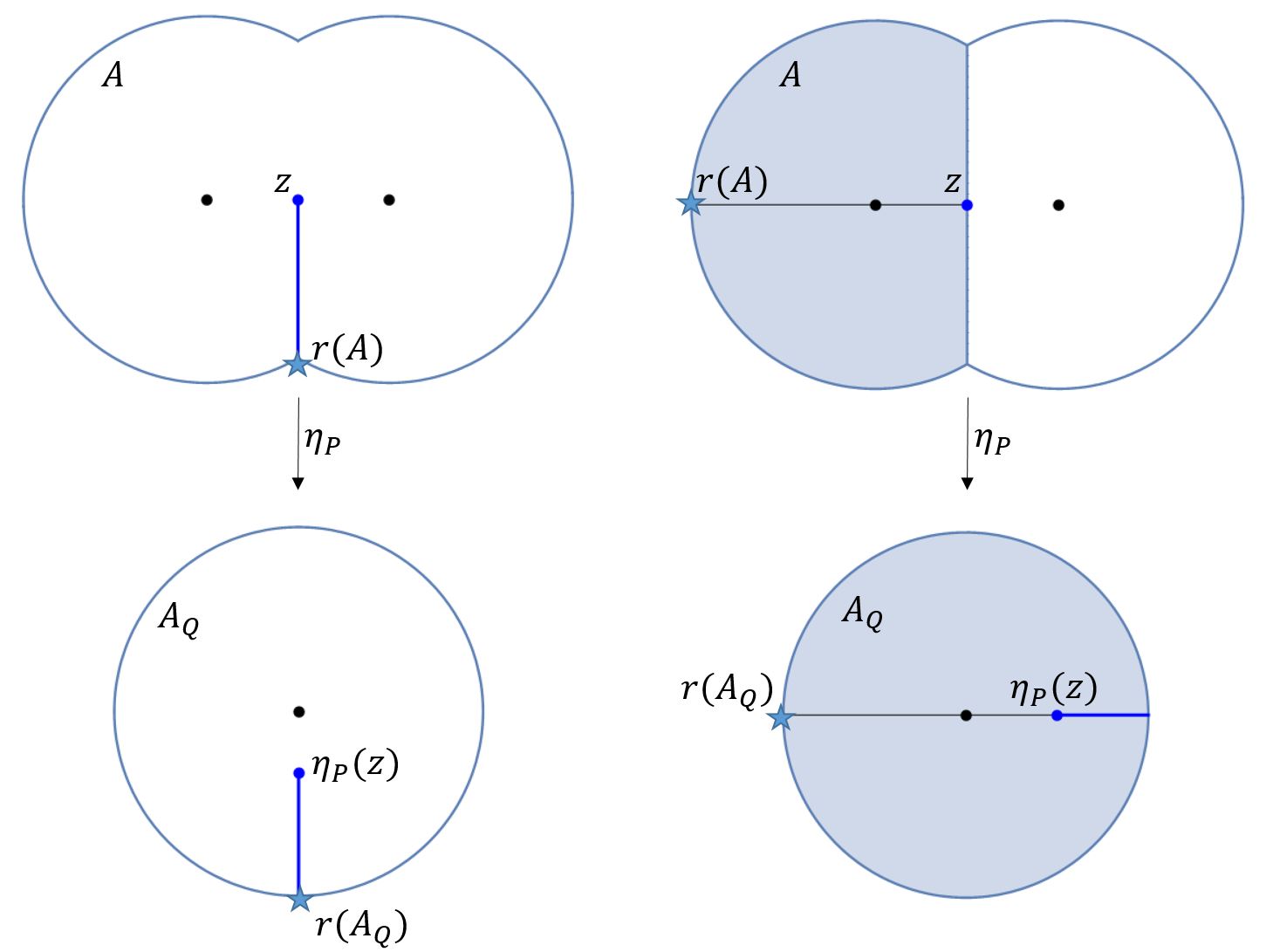}
  \caption{\small Left: case (1) of Theorem \ref{t:XPinbub}.
  Right: case (2) of Theorem \ref{t:XPinbub}.
  These illustrations are schematic (they show the right topology but not the right geometry of $A$ and $A_Q$).
  In both cases, the top figure shows $A$ and the bottom figure shows $A_Q$.
  The root points $r(A)$ and $r(A_Q)$ are marked by stars.
  Left: the set $X(P)\cap A$ is shown as the vertical segment connecting $r(A)$ with $z$,
   and the $\eta_P$-image of this set is shown as the vertical segment connecting $r(A_Q)$ with $\eta_P(z)$.
  Right: the set $X(P)\cap A$ is the shaded region of the top figure,
   and the $\eta_P$-image of it is the shaded region of the bottom figure.
  }\label{fig:53}
\end{figure}

\begin{proof}
Since $P^n:A\to P^n(A)$ is a homeomorphism mapping points of $Y(P)$ to points of $Y(P)$ and vice versa,
 it is enough to consider the case $c\in A$ (then $z=c$).
Since $A$ is open, it follows that $A\cap Y(P)\ne\0$.
Recall that the multi-angle of $A$ is the multi-angle of its root point $r(A)$.
Since it takes  two radii to pass through a bubble, we may assume that
$\vec\al=(\al_0,\dots,\al_{2k})$ is the multi-angle of $A$.
Then there are two cases: (1) the multi-angle of any point in $A\cap Y(P)$ looks like $(\vec\al.\al_{2k})$,
or (2) some points in $A$ have multi-angles $(\vec\al.\al_{2k})$ while others have multi-angles
$(\vec\al.\al_{2k}\al_{2k+1})$. Consider these cases separately.

(1) In this case, $c$ also has multi-angle $\vec\al.(\al_{2k})$.
Since all points $A\cap Y(P)$ have polar angle $\al_{2k}$ with respect to $A$,
 the set $\ol A\cap X(P)$ is the legal arc from the root point of $A$ to $c$.
It is also clear that $\eta_P$ is defined and continuous on this legal arc.
The $\eta_P$-image of $\ol A\cap X(P)$ is a legal arc in the closure of the bubble of $Q$ corresponding to $A$.

(2) Choose a point of $A\cap Y(P)$ with multi-angle $\vec\al.(\al_{2k},\al_{2k+1})$;
evidently, $\al_{2k+1}\ne\al_{2k}$. Set $B=P(A)$, then $B\subset Y(P)$.
Let us describe the $P$-image of $A\cap Y(P)$ as a subset of $B$.
Let $R$ be the internal ray of $B$ containing the critical value $P(c)$.
Then $P(A\cap Y(P))$ includes the center of $B$ and all internal rays of $B$ but $R$.
On $R$, a segment $T$ from $P(c)$ to the boundary of $B$ is not in $P(A\cap Y(R))$; other points of $R$ are in.
Call $T$ the \emph{special segment}.

The pullback of $T$ is an arc $T'\subset A$ that is the union of $\{c\}$
 and the two separatrices.
(Fig. \ref{fig:53}, top right, shows the arc $T'$ as the vertical segment through $z$.)
The arc $T'$ divides $A$ in two disjoint pullbacks of $B\sm T$, and $A\cap Y(P)$ is one of them.
The set $\ol A\cap X(P)$ contains a unique $P$-preimage of the center of $B$ and (initial segments of)
rays of all arguments emanating from this point; all rays but one extend to $\d A$, and one exceptional ray
crashes into $c$ and then splits into two branches (the separatrices).
Here by rays we mean integral curves of the radial vector field $v$ in $A$.
The set $\d(A\cap Y(P))\cap A$ 
 equals $T'$.
Clearly, the map $\eta_P$  extends to the separatrices.
The image $\eta_P(A\cap X(P))$ coincides with the entire bubble of $Q$ corresponding to $A$.
In this $Q$-bubble one radial segment (from $\eta_P(c)$ to the boundary of the bubble) is covered twice.
Otherwise, the map is one-to-one.
\end{proof}

Here, a \emph{terminal segment} of an internal ray of $A$ means a segment from some point in the ray to $\d A$.
Part $(2)$ of Theorem \ref{t:XPinbub} describes the map $\eta_P$ in the case $c\in X(P)\sm J(P)$.
Note that the map is not monotone in this case as it double folds an arc on the boundary of $X(P)$.

\subsection{A separation property}
Suppose that $W$ is a Siegel wedge of $P$.
Consider the corresponding wedge $W_Q$ in the dynamical plane of $Q$.
Such wedge is called \emph{$P$-adapted} (this notion depends on the choice of $P$).
Say that a $P$-adapted wedge $W_Q$ \emph{separates} a point $x$ from a point $x'$ if
 $x\in W_Q$ and $x'\notin \ol W_Q$.
This relation is symmetric: the wedge $\C\sm\ol W_Q$ is also $P$-adapted, and
 it separates $x'$ from $x$.

Note that, since $K(Q)$ is locally connected, any two points of $K(Q)$ can be connected by a legal arc.
The set $\eta_P(Y(P))\subset K(Q)$ is \emph{legal convex}, that is,
 any two points of this set can be connected by a legal arc lying entirely in this set.

\begin{lem}
  \label{l:regconv}
  The set $\ol{\eta_P(Y(P))}\subset K(Q)$ is legal convex.
\end{lem}

\begin{proof}
  This follows from a more general observation: the closure of a legal convex set is a legal convex set.
Indeed, if $x_n$, $x'_n\in K(Q)$ are two sequences converging to $x$, $x'$, respectively,
then the legal arc from $x_n$ to $x'_n$ converges to the legal arc from $x$ to $x'$.
\end{proof}

Separation of points within the closure of the same bubble of $Q$ is given by the following lemma.

\begin{lem}
\label{l:sep-adapt}
Let $A_Q\ne\Delta(Q)$ be a bubble of $Q$, and $x\in\d A_Q\cap\ol{\eta_P(Y(P))}$ be a point different from the root point of $A_Q$.
Then, for any other point $x'\in\d A_Q$, there is a $P$-adapted wedge $W_Q$ with root point in the center of $A_Q$
 such that $W_Q$ separates $x$ from $x'$.
Any pair of different points in $\d\Delta(Q)$ is also separated by a $P$-adapted wedge unless
 $c\in\d\Delta(P)$ and $P^k(c)=1$ for some $k\ge 0$.
\end{lem}

\begin{proof}
Suppose first that $A_Q\ne\Delta(Q)$ and $x\ne r(A_Q)$ is a point of $\d A_Q\cap\ol{\eta_P(Y(P))}$.
The legal arc $I_x$ in $K(Q)$ from $0$ to $x$ lies in $\ol{\eta_P(Y(P))}$, by Lemma \ref{l:regconv}.
By Lemma \ref{l:XPbub} and since $x$ is not the root point of $A_Q$,
 the bubble $A_Q$ corresponds to some legal bubble $A$ of $P$.
Since $x\in\d A_Q\cap\ol{\eta_P(Y(P))}$ is not equal to $r(A_Q)$, then case (2) of
Theorem \ref{t:XPinbub} holds. Hence the $\eta_P$-image of $\ol A\cap Y(P)$ is $\ol A_Q$ except for
 at most a subarc of $\ol R$ (where $R$ is an internal ray of $A_Q$) not reaching the center of $A_Q$.
Points of $\d A_Q$ that are root points of other bubbles attached to $A_Q$ are dense in $\d A_Q$.
All these points except possibly one are in $\eta_P(Y(P))$.
Therefore, there is a pair of such points $b$, $b'$ in $\d A_Q$ that separate $x$ from $x'$.
The legal arcs from the center $a$ of $A_Q$ to $b$ and $b'$ can be extended to periodic Siegel rays.
Moreover, there are periodic Siegel rays $\Ga_b$ and $\Ga_{b'}$ for $P$ such that
 the wedge $W$ bounded by $\Ga_b$ and $\Ga_{b'}$ corresponds to a wedge $W_Q$
 whose boundary intersects $\d A_Q$ at points $b$ and $b'$.
Thus there is a $P$-adapted wedge $W_Q$ that separates $x$ from $x'$, as desired.

Suppose now that $x$, $x'\in\d\Delta(Q)$ but either $c\notin\d\Delta(P)$ or
 $c\in\d\Delta(P)$ is never mapped to $1$ under iterates of $P$.
Then there are two iterated preimages $b$, $b'$ of $1$ separating $x$ and $x'$ in $\d\Delta(Q)$.
The corresponding points of $\d\Delta(P)$ are root points of legal bubbles attached to $\Delta(P)$,
 and the same argument as above works.
\end{proof}

The following is a more general separation property.

\begin{prop}
  \label{p:sep-w}
  A pair of distinct points $x$, $x'\in J(Q)\cap\ol{\eta_P(Y(P))}$ is separated by a $P$-adapted wedge,
  except when both $x$ and $x'$ are in $\d\Delta(Q)$, and $c\in\d\Delta(P)$ is eventually mapped to $1$.
\end{prop}

\begin{proof}
First assume that for some bubble $A_Q$ the point $x$ belongs to
$\ol{A_Q}$ and the legal arc $I$ from $x$ to $x'$ intersects $A_Q$.
Let $a$ be the center of $A_Q$.
Let $x''$ be the point of $\d A_Q$, where $I$ intersects $\d A_Q$; we necessarily have $x\ne x''$.
Lemma \ref{l:sep-adapt} is applicable to $x$ and $x''$ since at least one of these two points is different from $r(A_Q)$.
By Lemma \ref{l:sep-adapt}, there is a $P$-adapted wedge $W_Q$ with root point $a$ separating $x$ from $x''$.
Then $W_Q$ will also separate $x$ from $x'$ since the legal arc from $x'$ to $x''$ cannot intersect $\d W_Q$.

If neither $x$ nor $x'$ belongs to the boundary of a bubble,
then there are bubble rays $\Ac_Q$ and $\Ac'_Q$ (not necessarily periodic) landing at $x$ and $x'$, respectively
(since $J(Q)$ is locally connected, then any bubble ray lands).
Since $x\in\ol{\eta_P(Y(P))}$, it follows that there are infinitely many bubbles in $\Ac_Q$ intersecting $\eta_P(Y(P))$.
Then in fact all bubbles in $\Ac_Q$ intersect $\eta_P(Y(P))$, by Lemma \ref{l:XPbub}.
It follows that there is a bubble ray $\Ac$ for $P$ corresponding to $\Ac_Q$
 (recall that, by definition, a bubble ray for $P$ consists of legal bubbles).
Similarly, there is a bubble ray $\Ac'$ for $P$ corresponding to $\Ac'_Q$.
Take a bubble $B_Q\in \Ac_Q$ but not in $\Ac'_Q$.
By the above, $B_Q$ corresponds to some legal bubble $B$ of $P$.
Therefore, there exists a $P$-adapted wedge $W_Q$ separating $x$
 (equivalently, the point where the legal arc from the center of $B_Q$ to $x$ intersects $\d B_Q$)
 from the root point of $B_Q$.
Then $W_Q$ also separates $x$ from $x'$.
\end{proof}

 \subsection{Continuous extension of $\eta_P$}
In this section, we complete the proof of the following theorem.

\begin{thm}
  \label{t:etaPcont}
  The map $\eta_P:Y(P)\to K(Q)$ extends to a continuous map $\eta_P:X(P)\to K(Q)$.
Unless $c\in X(P)\sm J(P)$, this map is monotone.
\end{thm}

The extended map is denoted by the same letter.

\begin{proof}[Proof of Theorem \ref{t:etaPcont}, definition of $\eta_P$ and its continuity]
We start by proving that $\eta_P$ extends continuously to $\ol{Y(P)}$.
  Take $y\in\d Y(P)\sm Y(P)$; by Theorem \ref{t:XPinbub}, it is enough to assume that $y$ is not in a bubble.
We need to prove that, for all sequences $y_n\in Y(P)$ converging to $y$,
 the images $\eta_P(y_n)$ converge to the same limit.
Assume the contrary: $y_n$, $y'_n\in Y(P)$ are two sequences converging to $y$ such that
 $$
 \lim_{n\to\infty} \eta_P(y_n)=x\ne x'=\lim_{n\to\infty} \eta_P(y'_n).
 $$
Set $x_n=\eta_P(y_n)$ and $x'_n=\eta_P(y'_n)$.

By Proposition \ref{p:sep-w}, there is a $P$-adapted wedge $W_Q$ in the dynamical plane of $Q$
 that separates $x$ from $x'$ so that $x\in W_Q$ and $x'\notin\ol W_Q$.
Since $W_Q$ is open, $x_n\in W_Q$ for all large $n$.
By definition of a $P$-adapted wedge, $W_Q$ corresponds to some legal wedge $W$ for $P$.
Thus, $y_n\in W$ for large $n$, and 
 these $y_n$ are in some compact subset of $W$.
It follows that $y\in W$, hence also $y'_n\in W$ for large $n$.
We conclude that $x'_n\in W_Q$ for large $n$, therefore, $x'\in\ol W_Q$, a contradiction.

Suppose now that $X(P)\ne\ol{Y(P)}$, that is, there is a parabolic cycle in $\ol{Y(P)}$.
Every point $z\in X(P)\sm\ol{Y(P)}$ belongs to the closure of a parabolic domain at a parabolic point $a_z$.
Set $\eta_P(z)=\eta_P(a_z)$.
The extension thus defined is continuous.
\end{proof}

As usual, \emph{fibers} of $\eta_P$ are defined as preimages of points under $\eta_P$.
We now address the issue of connectedness of fibers.

\begin{lem}
  \label{l:intW}
A nonempty intersection of finitely many Siegel wedges for $P$ is connected and has a connected intersection with $X(P)$.
\end{lem}

\begin{proof}
Define a \emph{quasi-chord} in $\C$ as the image of $\R$ under some proper
 topological embedding into $\C$.
Then a quasi-chord divides the plane $\C$ into two open unbounded regions.
Consider a Siegel wedge $W$ bounded by two Siegel rays $\Sigma$, $\Sigma'$.
Let $b$ be the root point of $W$, that is, the point where $\Sigma$ and $\Sigma'$ branch off.
The boundary of $W$ is a quasi-chord
 containing $b$ as well as pieces of $\Sigma$ and $\Sigma'$ connecting $b$ to infinity.
Define a \emph{Siegel quasi-chord} as a quasi-chord in $\C$ that is the boundary of some Siegel wedge.

Suppose now that $U$ is a nonempty intersection of finitely many Siegel wedges.
Then, clearly, $U$ is an unbounded connected and simply connected domain
 whose boundary is a union of finitely many quasi-chords.
Each boundary quasi-chord of $U$ has its own root point.
Either $0\in\ol U$, or there is a unique boundary quasi-chord of $U$ whose root point $b_U$
 has multi-angle different from $()$.
The lemma follows from the observation that any point of $Y(P)\cap U$ can be connected
 to $0$ or $b_U$ by a legal arc lying in $Y(P)$.
\end{proof}

\begin{proof}[Proof of Theorem \ref{t:etaPcont}, the monotonicity part]
Take $x_Q\in\eta_P(X(P))$; consider the fiber $\eta_P^{-1}(x_Q)$.
If $x_Q$ is in a bubble of $Q$, then the fiber of $x_Q$ is a singleton by Theorem \ref{t:XPinbub}.
Thus we may assume that $x_Q\in J(Q)$.
First suppose that $x_Q\in\d\Delta(Q)$ and $c\in\d\Delta(P)$ is eventually mapped to $1$.
Note that $\eta_P:\d\Delta(P)\to\d\Delta(Q)$ is a homeomorphism.
Hence there is a unique point $x\in\d\Delta(P)$ such that $\eta_P(x)=x_Q$.
Suppose that some point $y\notin\ol\Delta(P)$ is mapped to $x_Q$.
Since $y\in X(P)$, there is a sequence $y_n\in Y(P)$ converging to $y$; let $I_{y_n}$ be the legal arcs from $0$ to $y_n$.
Obviously, $y\in J(P)$, and it is possible to arrange that all $y_n\in J(P)$ as well.
Passing to a subsequence, we may assume that $I_{y_n}$ converge to a continuum $C_y\ni y$.

If $y$ is on the boundary of a legal bubble $A$ of $P$, then $A$ can be chosen so that $y\ne r(A)$.
Then $\eta_P(y)$ is on the boundary of the bubble $A_Q$ of $Q$ corresponding to $A$, and $\eta_P(y)\ne r(A_Q)$.
A contradiction with $\eta_P(y)=x$.
Thus we assume that $y$ is not on the boundary of a bubble.

Let $A_n$ be the bubble attached to $\Delta(P)$ and such that $I_{y_n}\cap A_n\ne\0$.
If there are only finitely many different bubbles $A_n$, then, passing to a subsequence,
 we may assume that all $A_n$ are the same bubble $A$; set $A_Q$ to be the corresponding bubble of $Q$.
The intersection $I_{y_n}\cap A$ is the union of two internal rays landing at $r(A)\in\Delta(P)$ and $b_n\ne r(A)$.
If $b_n\to r(A)$, then we can replace $C_y$ with $C^*_y=C_y\sm A$.
The latter is a continuum containing $y$, and $C^*_y\sm\{x\}$ is disjoint from all bubbles.
If $b_n\not\to r(A)$, then, passing to a subsequence, assume that $b_n\to b\ne r(A)$.
It follows that $\eta_P(C_y\sm \ol A)$ is attached to $A_Q$ at $\eta_P(b)$; it does not accumulate on $\d\Delta(Q)$.
A contradiction with $x_Q=\eta_P(y)\in\d\Delta(Q)$.
Finally, if there are infinitely many pairwise different $A_n$s, then we set $C^*_y=C_y$.
In any case, $C^*_y$ is a continuum containing $y$ such that $C^*_y\sm\{x\}$ is disjoint from all bubbles.
It follows that $\eta_P(C^*_y)=\{x_Q\}$.
Since $\eta_P^{-1}(x_Q)$ is the union of $\{x\}$ and $C^*_y$ over all $y\in\eta_P^{-1}(x_Q)\sm\{x\}$,
 the fiber $\eta_P^{-1}(x_Q)$ is connected.

Now assume that $x_Q\notin\d\Delta(Q)$ or that $c\notin\d\Delta(P)$ or that $c\in\d\Delta(P)$ is never mapped to $1$.
Let $Z_Q$ be the intersection of all $P$-adapted wedges containing $x_Q$.
By the separation property, Proposition \ref{p:sep-w}, we have $Z_Q\cap J(Q)=\{x_Q\}$.
Apart from $x_Q$, the set $Z_Q$ may include certain external rays of $Q$ as well as certain internal rays in bubbles of $Q$.
Thus $Z_Q\cap K(Q)$ is the union of $\{x_Q\}$ and one or two internal rays in bubbles $A_Q$ such that $x\in\d A_Q$.
Moreover, every such bubble may intersect $Z_Q$ by at most one internal ray.

Consider the full preimage $Z=\eta_P^{-1}(Z_Q\cap K(Q))$.
This is the union of $\eta_P^{-1}(x_Q)$ and one or two internal rays in legal bubbles $A$ of $P$.
Every such bubble may intersect $Z$ by at most one internal ray.
It follows that connectedness of $Z$ will imply connectedness of the fiber $\eta_P^{-1}(x_Q)$.
The rest of the proof deals with connectedness of $Z$.

Take any $x\in\eta_P^{-1}(x_Q)\cap J(P)$.
Note that $Z$ is the intersection of all Siegel wedges of $P$ containing $x$ with $X(P)$.
Moreover, it is enough to intersect countably many Siegel wedges $W_1$, $\dots$, $W_n$, $\dots$:
$$
Z=X(P)\cap\bigcap_{n=1}^\infty W_n=X(P)\cap\bigcap_{n=1}^\infty U_n,\quad U_n=W_1\cap\dots\cap W_n.
$$
The sets $X(P)\cap U_n$ are connected by Lemma \ref{l:intW} and form a nested sequence.
Therefore, their intersection is also connected.
\end{proof}

\section{The parameter maps $\Phi^c_\la$ and $\Phi_\la$}
\label{s:par}
Consider a map $P=P_c\in\Cc^c_\la$.
Suppose that $c\in X(P)$.
Define $\Phi^{c}_\la(P)$ as $\eta_{P}(c)$.
Let $\Dc^c_\la$ denote the domain of the map $\Phi^c_\la$, that is, the set of all $P\in\Cc^c_\la$ such that $c\in X(P)$.
Observe that $\Zc^c_\la\subset\Dc^c_\la$.

\subsection{Immediate renormalization}
Recall the notions of a polynomial-like map and an immediate renormalization.
Write $U\Subset V$ if $\ol{U}\subset V$.
Let $U\Subset V$ be Jordan disks in $\C$.
The following classical definitions are due to Douady and Hubbard \cite{DH-pl}.

\begin{dfn}[Polynomial-like maps \cite{DH-pl}]\label{d:pl}
Let $f:U\to V$ be a proper holomorphic map.
Then $f$ is said to be \emph{polynomial-like} (PL).
By definition, a \emph{quadratic-like} (QL) map is a PL map of degree two.
The \emph{filled Julia set} $K(f)$ of $f$ is defined as the set of points in $U$, whose forward $f$-orbits stay in $U$.
\end{dfn}

Similarly to polynomials, \emph{the set $K(f)$ is connected if and only if all critical points of $f$ are in $K(f)$}.
The following is a greatly simplified and weakened version of a much stronger classical theorem of Douady and Hubbard \cite{DH-pl}.

\begin{thm}[PL Straightening Theorem \cite{DH-pl}] \label{t:DH-pl}
A PL map $f:U\to V$ is topologically conjugate to a polynomial of the same degree restricted on a
Jordan neighborhood of its filled Julia set.
\end{thm}

Theorem \ref{t:thmb} below appears to be a folklore result.
It is formally proved, e.g., in \cite{BOPT16} (Theorem B).

\begin{thm}
\label{t:thmb}
  Let $P:\C\to\C$ be a polynomial, and $Y\subset\C$ be a full
  $P$-invariant continuum. The following assertions are equivalent:
\begin{enumerate}
 \item the set $Y$ is the filled Julia set of some polynomial-like
     map $P:U^*\to V^*$ of degree $k$,
 \item $Y$ is a component of the set $P^{-1}(P(Y))$, and, for every
     attracting or parabolic point $y$ of $P$ in $Y$, the
     immediate attracting basin of $y$ or the union of all parabolic domains
     at $y$ is a subset of $Y$.
\end{enumerate}
\end{thm}

A cubic polynomial $P\in \Cc^c_\la$ is \emph{immediately renormalizable} if
$P:U\to V$ is a QL map for some $U$, $V$.

\begin{prop}
  \label{p:Dc}
  Suppose that $P\in\Cc^c_\la\sm\Dc^c_\la$.
Then $P$ is immediately renormalizable with $X(P)$ being the corresponding quadratic-like Julia set.
\end{prop}

\begin{proof}
Since $P\in\Cc^c_\la\sm\Dc^c_\la$, then $c\notin X(P)$.
The set $X(P)$ is compact; also, it is easy to see that $X(P)$ is
a component of $P^{-1}(X(P))$ (it suffices to consider the set $Y(P)$).
There are no parabolic periodic points of $P$ in $X(P)$;
 otherwise $c$ would be in one of the parabolic domains added to $X(P)$.
By Theorems \ref{t:DH-pl} and \ref{t:thmb}, there is a Jordan domain $U\supset X(P)$ such that
$P:U\to P(U)$ is a quadratic-like map whose filled Julia set coincides with $X(P)$.
\end{proof}

Recall that the set $\Pc^c_\la$ is the subset of $\Cc^c_\la$ consisting of polynomials that can be approximated by
sequences $P_{n}\in\Cc^c_{\la_n}$ with $|\la_n|<1$ and both critical points of $P_n$ in the immediate basin of $0$.

\begin{cor}
  \label{c:PcDc}
  The set $\Pc^c_\la$ is a subset of $\Dc^c_\la$.
\end{cor}

\begin{proof}
  We will prove an equivalent statement: if $P\in\Cc^c_\la\sm\Dc^c_\la$, then $P\notin\Pc^c_\la$.
By Proposition \ref{p:Dc}, there is a quadratic-like map $P:U\to V$ with filled Julia set $X(P)$,
 and we may choose $U$ and $V$ so that $c\notin V$.
There is $\eps>0$ with the following property: if a cubic polynomial $f$ is $\eps$-close to $P$,
 then, setting $U_f$ to be a component of $f^{-1}(V)$ containing $0$, we obtain
 a quadratic-like map $f:U_f\to V$.
This follows, e.g., from Lemma \ref{l:Hcont}.
On the other hand, since $|\la|=1$, we can choose $f$ in $\Cc_\mu$ with $|\mu|<1$.
The filled Julia set $K^*_f$ of the quadratic-like map $f:U_f\to V$ then contains the immediate attracting basin of $0$.
It follows from the Douady--Hubbard straightening theorem \cite{DH-pl} that $K^*_f$ is a Jordan disk
 on which $f$ is two-to-one.
In particular, it is impossible that $f$ is in the principal hyperbolic component, and $P\notin\Pc^c_\la$,
 as claimed.
\end{proof}

\subsection{Continuity}
It will be established in this section that $\Phi^c_\la$ is continuous.
Recall that $P_1$ is the polynomial in $\C^*_\la$ such that $c=1$ is a multiple critical point.
The next lemma deals with continuity of $\Phi^c_\la$ at $P_1$.

\begin{lem}
  \label{l:P1}
  Suppose that a sequence $P_{c_n}\in\Dc^c_\la$ converges to $P_1$ (so that $c_n\to 1$).
If $\eta_{P_{c_n}}(c_n)$ converges, then the limit is equal to $1$.
\end{lem}

\begin{proof}
  Assume the contrary: $c_{Q,n}=\eta_{P_{c_n}}(c_n)$ converges to a point $c_Q$ different from $1$.
Let $y_n\in Y(P_{c_n})$ be a sequence of points such that $y_n$ is very close to $c_n$, in particular, $y_n\to 1$.

Suppose first that $c_Q\in \d\Delta(Q)$; let $(\alpha_0)$ be the multi-angle of $c_Q$.
Let $y$ be the point where the internal ray in $\Delta(P_1)$ of argument $\alpha_0$ lands.
There are two internal rays $R_Q$, $L_Q$ in $\Delta(Q)$ such that $\ol{R_Q\cup L_Q}$
 separates $1$ from $c_Q$ in $\ol\Delta(Q)$.
There is a simple unbounded curve $\Gamma_Q$ (a closed subset of $\C$ homeomorphic to $\R$)
 separating $1$ from $c_Q$ in $\C$ and such that $R_Q$, $L_Q\subset\Gamma_Q$.
Moreover, we can assume that $\Gamma_Q$ consists of internal rays in various bubbles,
 centers of those bubbles, landing points of those rays, a couple of repelling periodic points of $Q$,
 and a couple of external rays of $Q$ landing at these repelling points.
In other words, $\Gamma_Q$ is the union of two legal arcs from $0$ to repelling periodic points of $Q$
 and the external rays of $Q$ landing at these repelling points.
Thus $\Gamma_Q$ is similar to the boundary of an adapted wedge except that it is not adapted for $P_1$.

To $\Gamma_Q$, we want to assign a curve $\Gamma$ in the dynamical plane of $P_1$.
If is natural to require that $\Gamma$ consist of internal rays in various bubbles of $P_1$,
 centers of those bubbles, landing points of those rays, a couple of repelling periodic points of $P_1$, and
 a couple of external rays of $P_1$ landing at these repelling points.
Also, we require that there is a bijective correspondence between bubbles $A$ intersecting $\Gamma$
 and bubbles $A_Q$ intersecting $\Gamma_Q$ so that $A\cap\Gamma$ includes internal rays of the same arguments
 as $A_Q\cap\Gamma_Q$, adjacent bubbles correspond to adjacent bubbles, and $\Delta(P_1)$ corresponds to $\Delta(Q)$.
There is indeed such a curve $\Gamma$.
The existence of $\Gamma$ relies on the landing theorem, Theorem \ref{t:land}, which is also valid in our case.
On the other hand, $\Gamma$ as above is not unique.

The problem is, no matter which $\Gamma$ we choose, it is not stable.
If $P_1$ is replaced with $P_c$, where $c$ is close to $1$, then there is a set $\Gamma_c$ close to $\Gamma$.
However, $\Gamma_c$ may become disconnected (two adjacent bubbles through which $\Gamma$ goes may detach).
On the other hand, for each particular $c$ close to $1$, we may choose $\Gamma$ so that $\Gamma_c$ stays connected.
This amounts to choosing, for every bubble $A$ through which $\Gamma$ passes, a next bubble $A'$
 attached to the point of $\Gamma\cap A$ different from $r(A)$ so that
 $A'$ does not detach from $A$ in the dynamical plane of $P_c$.
Thus we choose both $\Gamma$ and $\Gamma_c$ depending on $c$.
These curves are close to each other (in the spherical metric), and both separate $1$ from $y$.
By the choice of $y_n$, it has multi-angle $(\alpha_{0,n},\dots)$ with $\alpha_{0,n}$ close to $\alpha_0$, for large $n$.
It follows that $y$ and $y_n$ are on the same side of $\Gamma_{c_n}$, and $1$ is on the other side.
Moreover, $\Gamma_{c_n}$ cannot accumulate on $1$.
A contradiction with $y_n\to 1$.

Suppose now that $c_Q\notin\d\Delta(Q)$; we may assume that $c_Q\in J(Q)$.
Let $(\alpha_0,\alpha_1,\alpha_2,\dots)$ be the multi-angle of $c_Q$.
Recall that $\alpha_2\ne\alpha_1=\alpha_0$.
Denote the multi-angle of $y_n$ as $(\alpha_{0,n},\alpha_{1,n},\alpha_{2,n},\dots)$.
Choose a large $n$ so that the first three terms in the multi-angle of $c_{Q,n}$
 are close to $\alpha_0$, $\alpha_1$, $\alpha_2$.
For this $n$, by continuity of $\eta_{P_{c_n}}$, we can also arrange that $\alpha_{i,n}$ are close to $\alpha_i$
 at least for $i=0,1,2$.
In particular, $\alpha_{i,n}\to\alpha_i$ as $n\to\infty$.
If $\alpha_0\ne 0$, then the same separation argument as above is applicable.
Thus we assume $\alpha_0=0$; it follows that $\alpha_{0,n}=0$ for large $n$.

Let $A_{Q,1}$ be the bubble of $Q$ with multiangle $(0)$.
Suppose that $x_Q$ is the landing point of the internal ray in $A_{Q,1}$ of argument $\alpha_2$.
There is an unbounded simple curve $\Gamma_Q$ that includes the center $o_{A_{Q,1}}$ of $A_{Q,1}$ and a pair
 of internal rays in $A_{Q,1}$ and that separates $x_Q$ from $1$.
It then also separates $c_Q$ from $1$.
We may assume that $\Gamma_Q\cap K(Q)$ is the union of two legal arcs from $o_{A_{Q,1}}$
 to certain repelling periodic points of $Q$ and the external rays of $Q$ landing at these repelling points.
In other words, $\Gamma_Q$ is almost as above except that it is now centered at $o_{A_{Q,1}}$ rather than $0$.
The rest of the proof is the same as above.
\end{proof}

Theorem \ref{t:contoutZ} completes the proof of Theorem \ref{t:main1}.

\begin{thm}
  \label{t:contoutZ}
  The map $\Phi^c_\la:\Dc^c_\la\to K(Q)$ is continuous.
\end{thm}

\begin{proof}
Take $P\in\Dc^c_\la$.
Suppose that $P_{c_n}\in\Dc^c_\la$ converge to $P=P_c$.
We show that $c_{Q,n}=\Phi^c_\la(P_n)$ converge to $c_Q=\Phi^c_\la(P)$.
If not, then by choosing a suitable subsequence, we may assume that $c_{Q,n}\to c'_Q\ne c_Q$.
Now consider several cases.

First, suppose that $c_Q$ belongs to a bubble $A_Q$ of $Q$.
By definition, $A_Q$ corresponds to a legal bubble $A$ of $P$ containing $c$.
The sequence $c_n$ converges to $c$.
The bubble $P(A)$ is stable, in particular, there is a unique bubble $B_n$ of $P_{c_n}$ close to $P(A)$, for large $n$.
Moreover, $B_n$ contains the critical value $P_{c_n}(c_n)$.
By Lemma \ref{l:Hcont} it follows that a component $A_n$ of $P_{c_n}^{-1}(B_n)$
 contains the critical point $c_n$ and is close to $A$.
All $A_n$ have the same multi-angle, thus they all correspond to $A_Q$.
Now, since both $c_Q$ and $c'_Q$ lie in the same bubble $A_Q$, they have different images $Q(c_Q)\ne Q(c'_Q)$.
On the other hand, $\eta_{\tilde P}(\tilde P(\tilde c))$ depends continuously on $\tilde P=P_{\tilde c}$ near $P$
 (because of the stability of $P(A)$).
It follows that $Q(c_{Q,n})=\eta_{P_{c_n}}(P_{c_n}(c_n))\to \eta_P(P(c))=Q(c)$.
On the other hand, $Q(c_{Q,n})\to Q(c'_Q)$ since $c_{Q,n}\to c'_Q$.
Thus we must have $c_Q=c'_Q$ in the considered case.

Suppose now that $c_Q\in J(Q)$ and either $c\notin\d\Delta(P)$ or $c\in\d\Delta(P)$ is never mapped to $1$ under $P$.
By Proposition \ref{p:sep-w}, there is an adapted wedge $W_Q$ separating $c'_Q$ from $c_Q$.
Since $W_Q$ is open, $c_{Q,n}\in W_Q$ for all sufficiently large $n$.
In fact, $c_{Q,n}$ even lie in some compact subset $C_Q$ of $W_Q$ for all large $n$.
Let $W$ be the Siegel wedge of $P$ corresponding to $W_Q$.
Since the boundary of $W$ is stable, there are Siegel wedges $W_n$ for $P_{c_n}$ close to $W$
 that correspond to the same $W_Q$.
It follows from $c_{Q,n}\in C_Q$ that $c_n\in W_n$ for large $n$.
Then $c\in\ol W$, a contradiction.

Finally, suppose that $c\in\d\Delta(P)$ and $P^k(c)=1$.
If $k=0$, then the theorem follows from Lemma \ref{l:P1}.
If $k>0$, then $P$ is conjugate to another polynomial $\tilde P\in\Zc^c_\la$ via a linear map that takes $c$ to $1$.
This conjugacy takes $1$ to a critical point $\tilde c$ of $\tilde P$ such that $\tilde P^k(1)=\tilde c$.
Since $\tilde c$ is never mapped to $1$ under $\tilde P$, the argument given above is applicable to $\tilde P$.
\end{proof}

\subsection{The unmarked map $\Phi_\la$}
We now study the unmarked map $\Phi_\la:\Pc_\la\to\tilde K(Q)$.
Recall that $\tilde K(Q)$ was defined as a model space obtained as a quotient of $K(Q)\sm\Delta(Q)$.
Namely, points of $\d\Delta(Q)$ that are $\ol\psi_Q$-images of complex conjugate points in $\uc$
 are identified in $\tilde K(Q)$.
Let $\pi:K(Q)\sm\Delta(Q)\to\tilde K(Q)$ be the quotient map.
The map $\pi\circ\Phi^c_\la:\Pc^c_\la\to\tilde K(Q)$ is then well defined and continuous.
It suffices to prove that $P_c$ and $P_{1/c}$ have the same images under $\pi\circ\Phi^c_\la$.
Then the map $\pi\circ\Phi^c_\la$ descends to a continuous map $\Phi_\la$ from $\Pc_\la$ to $\tilde K(Q)$,
 as is claimed in the following lemma.

\begin{lem}
  \label{l:welldef}
  The points $\Phi^c_\la(P_c)$ and $\Phi^c_\la(P_{1/c})$ have the same $\pi$-images in $\tilde K(Q)$.
\end{lem}

The proof of this lemma uses the notion of the
 \emph{angular difference} between two points $a$, $b\in\d\Delta(P)$.
This is the difference $\alpha-\beta\in\R/\Z$,
 where $a=\ol\psi_{\Delta(P)}(e^{2\pi i\alpha})$ and $b=\ol\psi_{\Delta(P)}(e^{2\pi i\beta})$.

\begin{proof}[Proof of Lemma \ref{l:welldef}]
  Note that $P_c$ and $P_{1/c}$ are affinely conjugate; the difference is only in how the critical points are marked.
Thus the angular difference between the two critical points in the boundary of the Siegel disk
 is the same up to a sign for $P_c$ and $P_{1/c}$.
It follows that $\Phi^c_\la(P_c)$ and $\Phi^c_\la(P_{1/c})$ have the same angular difference with $1$
 up to a sign in $\d\Delta(Q)$.
By definition, such points are identified in $\tilde K(Q)$.
\end{proof}

Lemma \ref{l:welldef} implies the Main Theorem.

\subsubsection*{Acknowledgements}
We are grateful to the referee for useful remarks.
The second named author was partially supported by NSF grant DMS-1807558.
The fourth named author has been supported by the Simons-IUM fellowship and
 by the HSE University Basic Research Program.

\end{document}